%% file: RenewalDependent.tex
\documentclass{amsart}

\usepackage{fullpage}
\usepackage{fancyhdr}
\usepackage[latin1]{inputenc}      
\usepackage[T1]{fontenc}
\usepackage{ae,aecompl}
\usepackage{amsmath, amsfonts, amssymb, amsthm}
\usepackage{mathtools}
\usepackage{graphicx}
\usepackage{caption}
\usepackage[labelformat=simple]{subcaption}
\usepackage{dsfont}
\usepackage{color}
\usepackage{enumitem}
\usepackage{microtype}
\usepackage[plainpages=false,pdfpagelabels]{hyperref}

\numberwithin{equation}{section}
\theoremstyle{plain}
\newtheorem{theorem}{Theorem}[section]
\newtheorem*{rtheo}{Renewal Theorem with dependent interarrival times}
\newtheorem{corollary}[theorem]{Corollary}
\newtheorem{lemma}[theorem]{Lemma}
\newtheorem{proposition}[theorem]{Proposition}
\theoremstyle{definition}
\newtheorem{definition}[theorem]{Definition}
\theoremstyle{remark}
\newtheorem{remark}[theorem]{Remark}
\newtheorem{example}[theorem]{Example}
\setlist[enumerate,1]{label=(\roman*),ref=(\roman*)}
\setlist[enumerate,2]{label=(\alph*),ref=(\Alph{enumi}\alph*)}

\renewcommand{\tilde}{\widetilde}
\renewcommand{\hat}{\widehat}
\renewcommand{\emptyset}{\varnothing}
\renewcommand{\rho}{\varrho}

\newcommand{\defeq}{\vcentcolon=}
\newcommand{\eqdef}{=\vcentcolon}
\newcommand{\om}{\omega}
\newcommand{\ee}{\mathrm{e}}
\newcommand{\e}{\varepsilon}
\newcommand{\mdim}{\delta}

\newcommand{\eigenf}{h}
\newcommand{\eigenv}{\gamma}

\newcommand{\omneu}{u}

\newcommand{\z}{\zeta}
\newcommand{\aaa}{a}

\newcommand{\PF}{\mathcal L}
\newcommand{\Laplace}{L}

\newcommand{\prob}{\Pi}
\newcommand{\im}{\mathbf{i}}
\newcommand{\inte}{\textup{int}}

\newcommand{\renfcn}{f}
\newcommand{\renfcnn}{g}
\newcommand{\codefun}{\xi}
\newcommand{\codefunc}{\chi}
\newcommand{\roz}{\rho(z)}

\renewcommand\thesubfigure{(\Alph{subfigure})}

\begin{document}

\title[Renewal thms.\ for dependent interarrival times]{Renewal theorems for a class of processes with dependent interarrival times and applications in geometry}

\author{Sabrina Kombrink}
\address{Sabrina Kombrink, Universit\"at zu L\"ubeck, Ratzeburger Allee 160, 23562 L\"ubeck, Germany}
\email{kombrink@math.uni-luebeck.de}
\thanks{Part of this work was supported by grant 03/113/08 of the \emph{Zentrale Forschungsf\"orderung, Universit\"at Bremen}.}

\begin{abstract}
 Renewal theorems are developed for point processes with interarrival times $W_n=\codefun(X_{n+1}X_n\cdots)$, where $(X_n)_{n\in\mathbb Z}$ is a stochastic process with finite state space $\Sigma$ and $\codefun\colon\Sigma_A\to\mathbb R$ is a H\"older continuous function on a subset $\Sigma_A\subset\Sigma^{\mathbb N}$. The theorems developed here unify and generalise the key renewal theorem for discrete measures and Lalley's renewal theorem for counting measures in symbolic dynamics. Moreover, they capture aspects of Markov renewal theory. 
 The new renewal theorems allow for direct applications to problems in fractal and hyperbolic geometry; for instance, results on the Minkowski measurability of self-conformal sets are deduced.
 Indeed, these geometric problems motivated the development of the renewal theorems.
\end{abstract}



\maketitle

\section{Introduction and statement of main results}\label{sec:intro}
\input{intro}

\section{Ruelle-Perron-Frobenius Theory -- the equilibrium distribution}\label{sec:preliminaries}
\input{subshifts}

\section{Renewal theorems}\label{sec:results}
\input{renewal}

\section{Proofs of the renewal theorems and their corollaries}\label{sec:proofs}
In Sec.~\ref{sec:PFanalyticp} we present essentials from complex Ruelle-Perron-Frobenius theory. These are used to prove the renewal theorems in Sec.~\ref{sec:proofsRT}.
\subsection{Analytic Properties of the Ruelle-Perron-Frobenius Operator}\label{sec:PFanalyticp}
\input{PFanalytic}
\subsection{Proof of the renewal theorems}\label{sec:proofsRT}
\input{renewalproof}

\input{corproof}

\bibliographystyle{alpha}
\bibliography{Literaturvz}

\end{document}

%% file: intro.tex
\subsection{Renewal Theory}
Renewal theorems have found wide applicability in various areas of mathematics (such as fractal and hyperbolic geometry), economics (such as queueing, insurance and ruin problems) and biology (such as population dynamics). 
They are concerned with waiting times in-between occurrences of a repetitive pattern connected with repeated trials of a stochastic experiment. In classical renewal theory, it is assumed that after each occurrence of the pattern, the trials start from scratch. This means that the trials which follow an occurrence of the pattern form a replica of the whole stochastic experiment. In other words, the \emph{waiting times} in-between successive occurrences of the pattern, also called \emph{interarrival times}, are assumed to be mutually independent random variables with the same distribution (see \cite[Ch.~XIII]{FellerI} and \cite{Feller}).
The classical renewal theorems have been extended in various ways and to various different settings. One such extension, which is of particular interest to us, is given by Markov renewal theory. 
By a \emph{Markov random walk}, we understand a point process for which the interarrival times $W_0,W_1,\ldots$ are not necessarily independent identically distributed (i.\,i.\,d.), but Markov dependent on a discrete Markov chain $(X_n)_{n\in\mathbb N_0}$. This means that $W_n$ is sampled according to the current and proximate values $X_n,X_{n+1}$ but is independent of the past values $X_{n-1},\ldots,X_{0}$.
In the present article, we drop the assumption that $(X_n)_{n\in\mathbb N_0}$ is a Markov chain and that $W_n$ is Markov dependent on $(X_n)_{n\in\mathbb N_0}$. Instead, we consider a time-homogeneous (i.\,e.\ stationary increments) stochastic process $(X_n)_{n\in\mathbb Z}$ with finite state space and time-set $\mathbb Z$ and extend to the setting that $W_n$ may depend on the current values $X_{n+1},X_{n}$ as well as on the whole past $X_{n-1},X_{n-2},\ldots$ of the stochastic process $(X_n)_{n\in\mathbb Z}$. The dependence of $W_n$ on $X_{n+1},X_n,\ldots$ is assumed to be described by a H\"older continuous function. 
This additionally allows us to treat situations with more strongly dependent interarrival times.
Before giving an outline of the theorems we want to provide references to the afore-mentioned.  
Since the literature on classical and Markov renewal theory is vast, we abstain from presenting a complete list but instead refer to the following monographs and fundamental articles, where further references can be found: \cite{FellerI,Feller,MarkovRTSurvey,Alsmeyer,Asmussen,RenewalProcesses}.

We let $\Sigma\defeq\{1,\ldots,M\}$, $M\geq 2$ denote the state space of the time-homo\-geneous stochastic process $(X_n)_{n\in\mathbb Z}$. 
The admissible transitions are assumed to be governed by a \emph{primitive} $(M\times M)$- incidence matrix $A$ of zeros and ones, meaning there exists $n\in\mathbb N$ such that all entries of $A^n$ are positive.
The set of \emph{one-sided infinite admissible paths} through $\Sigma$ consistent with $A=(A(i,j))_{i,j\in\Sigma}$ is defined by $\Sigma_A\defeq\{x\in\Sigma^{\mathbb N}\mid A(x_k,x_{k+1})=1\ \forall\, k\in\mathbb N\}$.
Elements of $\Sigma_A$ are interpreted as paths which describe the history of the process, supposing that the process has been going on forever. 
Given $x\in\Sigma_A$, we study the limiting behaviour as $t\to\infty$ of the renewal function 
\begin{align*}
   N(t,x)
   \defeq \mathbb E_x\left[\sum_{n=0}^{\infty}\tilde{\renfcn}_{X_n\cdots X_1 x}\left(t-\sum_{k=0}^{n-1}W_k\right)\right],
\end{align*}
where $\mathbb E_x$ is the conditional expectation given $X_0X_{-1}\cdots=x$, the family $\{\tilde{\renfcn}_y\colon\mathbb R\to\mathbb R\mid y\in\Sigma_A\}$ satisfies some regularity conditions (see Sec.~\ref{sec:RTneu}) and for $n=0$ we interpret $\tilde{\renfcn}_{X_n\cdots X_1 x}(t-\sum_{k=0}^{n-1}W_k)$ to be $\tilde{\renfcn}_x(t)$.
For instance, if $\tilde{\renfcn}_y\defeq\mathds 1_{[0,\infty)}$, then $N(t,x)$ gives the expected number of renewals in the time-interval $(0,t]$ given $X_0X_{-1}\cdots=x$.
A natural assumption in applications is that the recent history of $(X_n)_{n\in\mathbb Z}$ has more influence on which state will be visited next than the earlier history. This is reflected in our assumption that the function $\eta\colon \Sigma_A\to\mathbb R$ given by $\eta(ix) \defeq \log\mathbb P_x(X_{1}=i)$ is $\alpha$-H\"older continuous for some $\alpha\in(0,1)$ in the sense of Defn.~\ref{defn:continuous}. Here, $\mathbb P_x$ is the distribution corresponding to $\mathbb E_x$ and $i\in\Sigma$. 
 Note that $\mathbb P_x(X_1=i)\defeq \mathbb P(X_1=i\mid X_0X_{-1}\cdots=x)>0$ if $ix\in\Sigma_A$ by definition of $\Sigma_A$. 
 Similarly, we assume that there exists $\codefun\in\mathcal F_{\alpha}(\Sigma_A,\mathbb R)$ with $W_n=\codefun(X_{n+1}X_nX_{n-1}\cdots)$. Here, $F_{\alpha}(\Sigma_A,\mathbb R)$ denotes the class of real-valued $\alpha$-H\"older continuous functions on $\Sigma_A$, see Def.~\ref{defn:continuous}.
This notation allows us to evaluate the conditional expectation and express $N(t,x)$ in a deterministic way. For this, write $S_n\codefun\defeq\sum_{k=0}^{n-1}\codefun\circ\sigma^k$ for the \emph{$n$-th Birkhoff sum} of $\codefun$ with $n\in\mathbb N$ and $S_0 \codefun\defeq 0$.
Here, $\sigma\colon\Sigma_A\to\Sigma_A$ denotes the (left) shift-map on $\Sigma_A$ which is defined by $\sigma(\om_1\om_2\ldots)\defeq\om_2\om_3\ldots$ for $\om_1\om_2\cdots\in\Sigma_A$. Notice, $\sigma(\om)$ describes the path of the process prior to the current time. Then $\sum_{k=0}^{n-1} W_k=S_n\codefun(X_{n}X_{n-1}\cdots)$ and, for $x,y\in\Sigma_A$ with $\sigma^ny=x$, we have $\mathbb P(X_nX_{n-1}\cdots=y\mid X_0X_{-1}\cdots=x)=\exp(S_n\eta(y))$.
Thus, 
\begin{equation}\label{eq:intro:renfcn}
   N(t,x)
   =\sum_{n=0}^{\infty}\sum_{y\in\Sigma_A:\sigma^n y=x}\tilde{\renfcn}_y(t-S_n\codefun(y))\ee^{S_n\eta(y)}.
\end{equation}
From this, one can deduce the renewal-type equation
\begin{align*}
   N(t,x)
   =\sum_{y\in\Sigma_A:\sigma y=x} N(t-\codefun(y),y)\ee^{\eta(y)}+\tilde{f}_x(t),
\end{align*}
which justifies calling $N$ a renewal function.
Intuitively, interarrival times are non-negative and probabilities take values in $[0,1]$. However, when considering the deterministic form \eqref{eq:intro:renfcn}, we allow $\codefun$ to take negative values, provided there exists $n\in\mathbb N$ for which $S_n\codefun$ is strictly positive.
Note that this condition is equivalent to $\xi$ being co-homologous to a strictly positive function, see Rem.~\ref{rmk:strictlypos}. Moreover, we allow $\eta$ to be chosen freely from the class $\mathcal F_{\alpha}(\Sigma_A,\mathbb R)$.

Notice, the current setting extends and unifies the setting of established renewal theorems: In the context of classical renewal theory for finitely supported measures (in particular of the key renewal theorem), $\eta$ and $\codefun$ only depend on the first coordinate.
When $\eta$ and $\codefun$ only depend on the first two coordinates, we are in the setting of Markov renewal theory. 
If $\eta$ is the constant zero-function and $\tilde{\renfcn}_y(t)=\codefunc(y)\mathds 1_{[0,\infty)}(t)$, where $\codefunc\in \mathcal F_{\alpha}(\Sigma_A,\mathbb R)$ is non-negative, we are precisely in the setting of \cite{Lalley}, where renewal theorems for counting measures in symbolic dynamics were developed.
For more details on these connections, see Sec.~\ref{sec:RTcorollaries}.

To present the renewal theorems, we now introduce important quantities on which the asymptotic behaviour of $N(t,x)$ as $t\to\infty$ depends. 
We interpret $(X_{n+1},W_n)_{n\in\mathbb N_0}$ as a stochastic process with state space $\Sigma\times\mathbb R$ and define an analogue of a transition kernel $U\colon\Sigma_A\times(\mathcal P(\Sigma)\otimes\mathfrak B(\mathbb R))\to\mathbb R$ by
\begin{align*}
 U(x,\{j\}\times(-\infty,t]) 
 &\defeq \mathbb P(X_{n+1}=j,W_n\leq t\mid X_nX_{n-1}\cdots=x)\\
 &=\mathds 1_{(-\infty,t]}(\codefun(jx))\ee^{\eta(jx)}
\end{align*}
for $x\in\Sigma_A$, $j\in \Sigma$ and $t\in\mathbb R$.
Here $\mathcal P(\Sigma)$ denotes the power set of $\Sigma$ and $\mathfrak B(\mathbb R)$ denotes the Borel $\sigma$-algebra on $\mathbb R$.
 For given $x\in\Sigma_A$ and $j\in \Sigma$, we set $F_{x,j}(t)\defeq U(x,\{j\}\times(-\infty,t])$ which defines a distribution function $F_{x,j}$ of a finite measure with total mass $\exp(\eta(jx))$. Its Laplace transform at $s\in\mathbb R$ is given by 
$(\mathcal L F_{x,j})(s) \defeq \int_{-\infty}^{\infty}\ee^{-sT}\textup{d}F_{x,j}(T) =\ee^{\eta(jx)}\ee^{-s\codefun(jx)}$.
For a given  $s\in\mathbb R$ it can be extended to an operator $\PF_{\eta-s\codefun}\colon\mathcal C(\Sigma_A)\to \mathcal C(\Sigma_A)$ acting on the space $\mathcal C(\Sigma_A)$ of continuous complex-valued functions on $\Sigma_A$ (see Sec.~\ref{sec:cts}) by
    \begin{equation}\label{eq:def:PF} 
    \PF_{\eta-s\codefun} \codefunc(x)\defeq \sum_{y\in\Sigma_A:\sigma y=x} \codefunc(y) \ee^{\eta(y)-s\codefun(y)}. 
    \end{equation}
Then $\PF_{\eta-s\codefun}\mathds{1}_{[j]}(x)=(\PF F_{x,j})(s)$, where $[j]\defeq\{x=x_1x_0\cdots\in\Sigma_A\mid x_1=j\}$ denotes the \emph{cylinder set} of $j\in\Sigma$.
The operator $\PF_{\eta-s\codefun}$ is called the \emph{Ruelle-Perron-Frobenius} operator to the potential function $\eta-s\codefun$. 
The Ruelle-Perron-Frobenius operator is a linear operator which takes the role of the stochastic matrix in the classical Perron-Frobenius theory.
(If $\eta$ and $\codefun$ only depend on the first two letters, then $\PF_{\eta-s\codefun}$ has an interpretation of a matrix $(B_{ij}(s))_{i,j\in\Sigma}$ with entries $B_{ij}(s)\defeq \exp[\eta(ji)-s\codefun(ji)]$ if $i$ to $j$ is an allowed transition, and $B_{ij}(s)\defeq 0$ otherwise. Studying this matrix leads to Markov renewal theorems, see Sec.~\ref{sec:corMarkov}.)
The $n$-th iterates of the Ruelle-Perron-Frobenius operator given by
$\PF^n_{\eta-s\codefun}\codefunc(x) = \sum_{y\in\Sigma_A,\,\sigma^ny=x}\codefunc(y)\ee^{S_n(\eta-s\codefun)(y)}$
are closely related to the  renewal function $N$.
This becomes apparent by multiplying $N(t,x)$ with $\exp(-st)$ and writing $\tilde{\renfcnn}_y(t)\defeq\exp(-st)\tilde{\renfcn}_y(t)$, yielding
\begin{align}\label{eq:intro:N}
	\ee^{-st}N(t,x)
	=\sum_{n=0}^{\infty}\sum_{y\in\Sigma_A:\sigma^n y=x}\tilde{\renfcnn}_y(t-S_n\codefun(y))\ee^{S_n(\eta-s\codefun)(y)}.
\end{align}
We let $\mdim$ denote the unique value of $s$, for which $\PF_{\eta-s\codefun}$ has spectral radius one, see Prop.~\ref{thm:eigenvalueone}. 
This is exactly the value of $s$ at which the series in \eqref{eq:intro:N} jumps from being convergent to being divergent (if $\tilde{\renfcnn}_y$ are sufficiently tame, in particular bounded, which is ensured by regularity conditions on $\tilde{\renfcn}_y$, see Sec.~\ref{sec:RTneu}).
The unique eigenfunction of $\PF_{\eta-\mdim\codefun}$ to the eigenvalue one is denoted by $\eigenf_{\eta-\mdim\codefun}$. The associated equilibrium measure, i.\,e. the fixed-point measure of the dual operator $\PF_{\eta-\mdim\codefun}^*$ acting on the set of Borel probability measures supported on $\Sigma_A$, is denoted by $\nu_{\eta-\mdim\codefun}$. We define $\mu_{\eta-\mdim\codefun}$ through $\textup{d}\mu_{\eta-\mdim\codefun}/\textup{d}\nu_{\eta-\mdim\codefun}=\eigenf_{\eta-\mdim\codefun}$. This is the unique $\sigma$-invariant Gibbs measure for the potential function $\eta-\mdim\codefun$. 
For more details on these terms see Sec.~\ref{sec:preliminaries}, \cite[Thm.~2{.}16, Cor.~2{.}17]{Walters_convergence} and \cite[Theorem 1{.}7]{Bowen_equilibrium}.

It is well-known that the limiting behaviour of a renewal function depends on the values which the interarrival times assume. Loosely speaking if the possible values of the interarrival times form a spaced pattern we are in the lattice situation and otherwise in the non-lattice situation, see Defn.~\ref{defn:lattice}. 
This lattice -- non-lattice dichotomy also occurs in our main theorems. Briefly, their conclusion is that with $s=\mdim$ the term on the right hand side of \eqref{eq:intro:N} converges to a positive and finite constant depending on $x$ as $t\to\infty$, when $\codefun$ is non-lattice. When $\codefun$ is lattice, we only have convergence along subsequences, so that $N(t,x)$ is asymptotic to a periodic function. 
Here, we call two functions $f,g\colon\mathbb R\to \mathbb R$ \emph{asymptotic} as $t\to\infty$, written $f(t)\sim g(t)$ as $t\to\infty$, if for all $\e>0$ there exists $t^*\in\mathbb R$ such that for $t\geq t^*$ the value $f(t)$ lies between $(1-\e)g(t)$ and $(1+\e)g(t)$.
If the range of $g$ lies in $\mathbb R_{>0}$ then $f(t)\sim g(t)$ as $t\to\infty$ if and only if $\lim_{t\to\infty}f(t)/g(t)=1$.
Our renewal theoretic results are proved via an analytic approach using Ruelle-Perron-Frobenius theory inspired by \cite{Lalley}. They are summarised in the following theorem. 
We refer the reader to Sec.~\ref{sec:RTneu}, in particular, Thms.~\ref{thm:RT1} and \ref{thm:RT2}, for further details and the precise meaning of the required regularity conditions.
\begin{rtheo}
     Assume that the family 
     $\{\tilde{\renfcn}_x(t)\mid x\in\Sigma_A\}$ satisfies some regularity conditions.
   \begin{enumerate}
   \item If $\codefun$ is non-lattice, then 
   \begin{equation*}
    N(t,x)\sim\ee^{t\mdim}\underbrace{
\frac{\eigenf_{\eta-\mdim \codefun}(x)}{\int \codefun\textup{d}\mu_{\eta-\mdim \codefun}}\int_{\Sigma_A} \int_{-\infty}^{\infty}\ee^{-T\mdim}\tilde{\renfcn}_y(T)\textup{d}T\textup{d}\nu_{\eta-\mdim \codefun}(y)
}_{\eqdef G(x)}
   \end{equation*}
   as $t\to\infty$, uniformly for $x\in\Sigma_A$.
   \item Assume that $\codefun$ is lattice. Then there exists a periodic function $\tilde{G}_x$ 
such that
  \begin{align*}
    N(t,x)
    \sim \ee^{t\mdim}\tilde{G}_x(t).
  \end{align*}
  \item\label{it:averagethm} We always have 
  \begin{equation*}
  \lim_{t\to\infty}t^{-1}\int_0^{t}\ee^{-T\mdim}N(T,x)\textup{d}T=G(x).
  \end{equation*}
  \end{enumerate}  
\end{rtheo}

\subsection{An Application -- Minkowski content of fractal sets}\label{sec:applications}
A direct application of the renewal theorems with dependent interarrival times in geometry, namely existence of the Minkowski content of self-conformal sets, is the focus of the current section.  It is this application which led to developing the renewal theorems. 
Here, we focus on the general ideas as to how to apply the respective renewal theorem and refer to a forthcoming article by the author for the geometric details (especially concerning the approximation arguments) and extensions to broader classes of fractal sets.

Let $B\subset\mathbb R^d$ denote a bounded subset of the $d$-dimensional Euclidean space $(\mathbb R^d,\|\cdot\|_2)$.
We wish to understand the asymptotic behaviour of the \emph{($\ee^{-t}$)-parallel volume} $\lambda_d(B_{\ee^{-t}})$ of $B$ as $t\to\infty$. Here, $B_{\ee^{-t}}\defeq\{x\in\mathbb R^d\mid\inf_{b\in B}\|b-x\|_2\leq \ee^{-t}\}$ for $t\in\mathbb R$ and $\lambda_d$ denotes the $d$-dimensional Lebesgue measure. 
More precisely, we study existence of the \emph{Minkowski content}
\[
  \mathcal{M}(B)\defeq\lim_{t\to\infty}\ee^{t(d-D)}\lambda_d(B_{\ee^{-t}})
\]
of $B$ and determine its value when the limit exists.
The definition of the Minkowski content implicitly assumes existence of the \emph{Minkowski dimension} $D\defeq d+\lim_{t\to\infty}t^{-1}\log\lambda_d(B_{\ee^{-t}})$, which coincides with the box-counting dimension, see \cite[Prop.~3{.}2]{Falconer_Foundation}. If $D\in\mathbb N$ and $\lambda_D(B)>0$ then $\mathcal M(B)=\lambda_D(B)$. Otherwise $\lambda_d(B)=0$ and the Minkowski content, when it exists, can be interpreted as the $D$-dimension\-al volume of $B$, giving a substitute of the notion of volume for non-integer dimensions. Besides this geometric relevance, the Minkowski content plays a major role in the Weyl-Berry conjecture, which is concerned with spectral asymptotics of Laplace operators on domains with irregular boundaries, see \cite{Berryresonators,Berrywave,LapPom}. These are two of the reasons why the Minkowski content has attracted much attention in recent years (see e.\,g.\ \cite{survey} for a more in depth introduction).

\begin{figure}[t]
   \begin{subfigure}[c]{0.3\textwidth}
   \includegraphics[width=\textwidth]{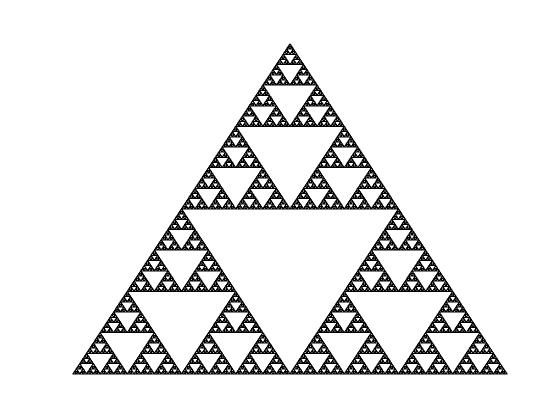}
   \subcaption{}
   \label{fig:Sierp_gasket}
   \end{subfigure}
   \quad
   \begin{subfigure}{0.3\textwidth}
   \includegraphics[width=\textwidth]{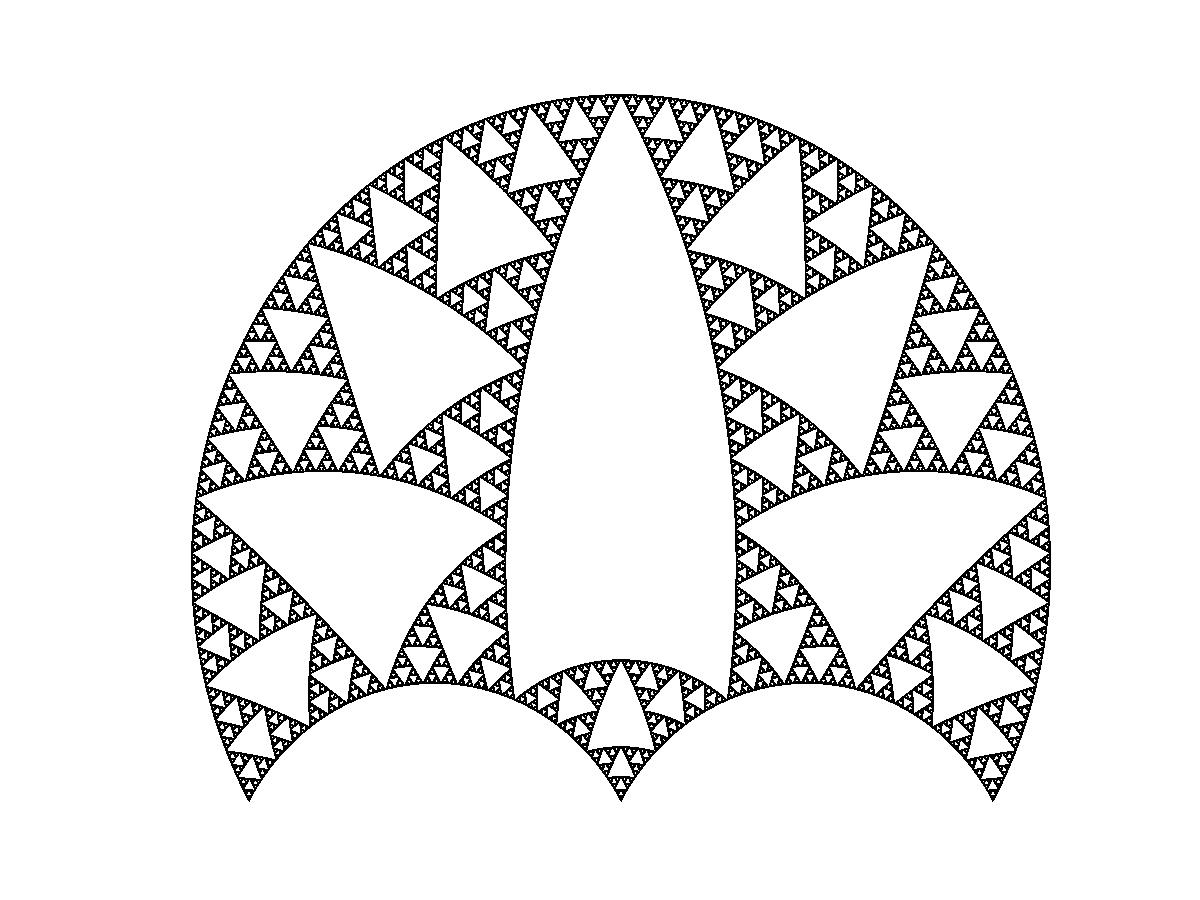}
   \subcaption{}
   \label{fig:bat}
   \end{subfigure}
   \quad
   \begin{subfigure}{0.3\textwidth}
   \includegraphics[width=\textwidth]{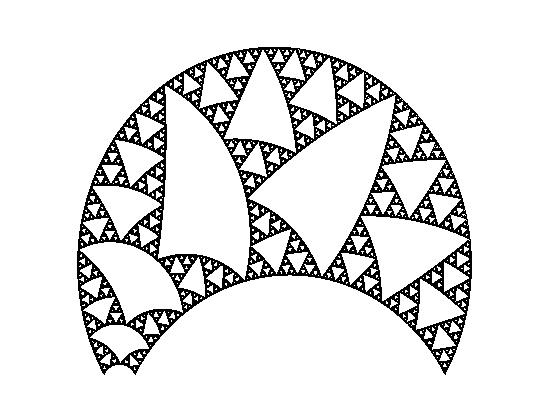}
   \subcaption{}
   \label{fig:moeb1}
   \end{subfigure}
   \caption{The self-conformal sets from Ex.~\ref{ex:pictures}: 
   \subref{fig:Sierp_gasket} The Sierpinski gasket (self-similar).
   \subref{fig:bat} and \subref{fig:moeb1} Conformal images of the Sierpinski gasket (self-conformal but not self-similar).}
   \label{fig:fractals}
\end{figure}

For defining self-conformal sets let $X\subset\mathbb R^d$ be a non-empty compact set with $X=\overline{\inte X}$. Define $\Phi\defeq\{\phi_1,\ldots,\phi_M\colon X\to X\}$ to be an iterated function system (IFS) consisting of contractive conformal $\mathcal C^{1+\alpha}$-diffeo\-morphisms $\phi_i$ with $M\geq 2$, $\alpha\in(0,1)$. 
Let $F$ be the unique compact non-empty set satisfying $F=\bigcup_{i=1}^M \phi_i F$. It exists due to \cite{Hutchinson} and is called the \emph{self-conformal set} associated with $\Phi$. 
$F$ is called \emph{self-similar} in the special case of $\phi_i$ being \emph{similitudes}, i.\,e.\ if $\|\phi_i(x)-\phi_i(y)\|_2=c_i\|x-y\|_2$ for all $x,y\in X$ and $i\in\{1,\ldots,M\}\eqdef\Sigma$ with $c_i\in(0,1)$.
Self-conformal sets provide prominent examples of fractal sets, see Ex.~\ref{ex:pictures} and Fig.~\ref{fig:fractals}.

\renewcommand\thesubfigure{}
\begin{figure}[t]
	\begin{subfigure}{0.25\textwidth}
		\includegraphics[width=\textwidth]{Sierpgasket}
		\includegraphics[width=\textwidth]{bat}
		\includegraphics[width=\textwidth]{t01}
		\caption{$F$}
	\end{subfigure}
	\hspace{-0.25cm}
	\begin{subfigure}{0.25\textwidth}
		\includegraphics[width=\textwidth]{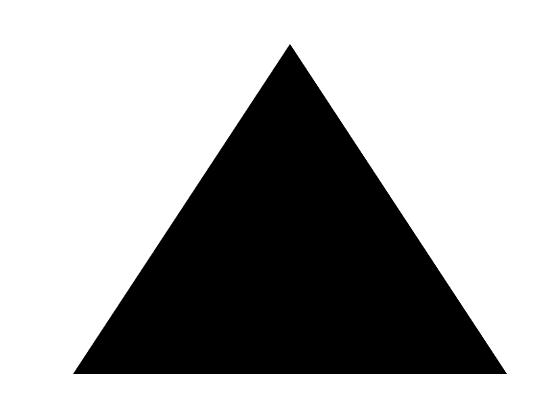}
		\includegraphics[width=\textwidth]{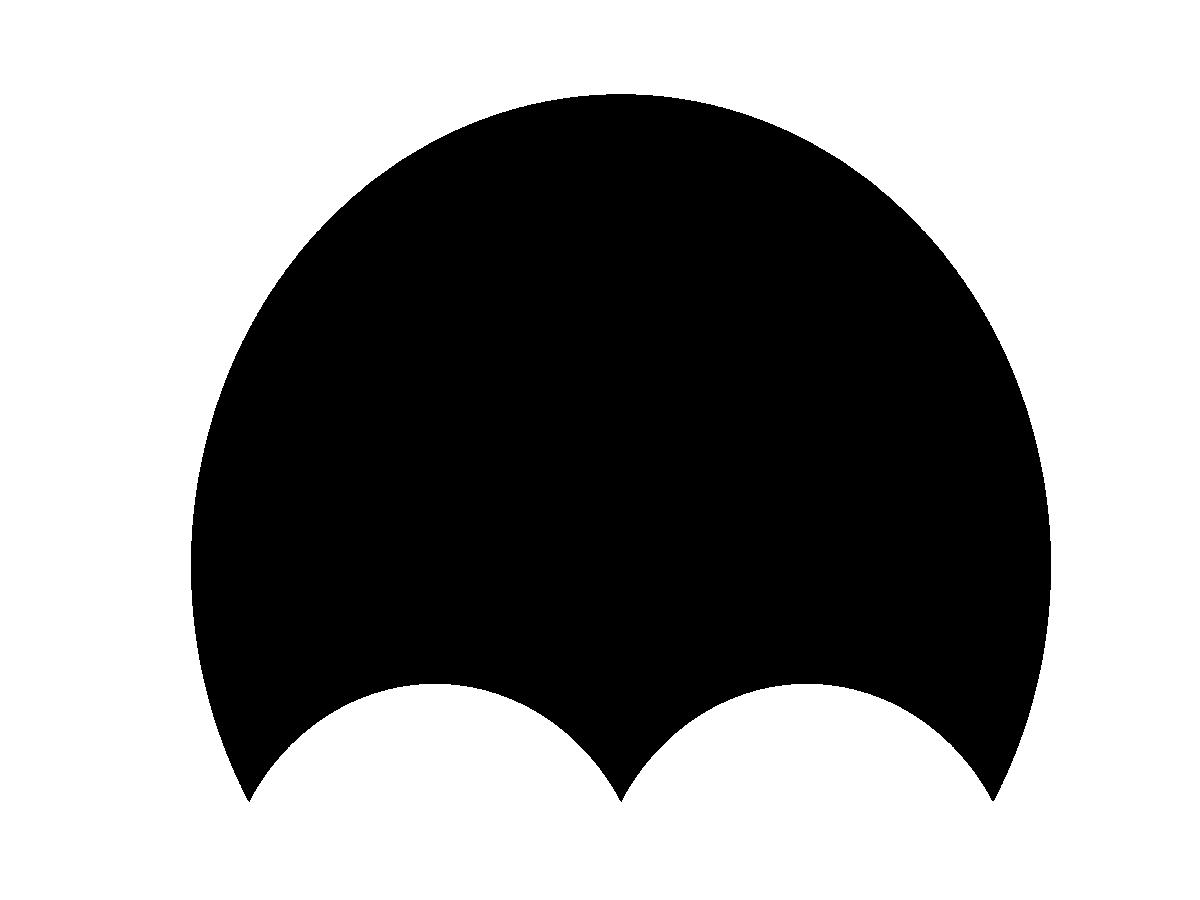}
		\includegraphics[width=\textwidth]{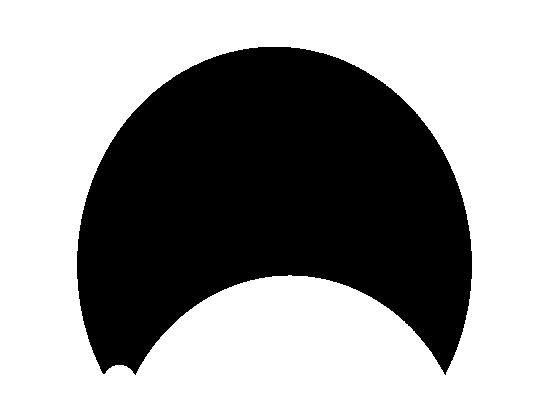}
		\caption{$O$}
	\end{subfigure}
	\hspace{-0.25cm}
	\begin{subfigure}{0.25\textwidth}
		\includegraphics[width=\textwidth]{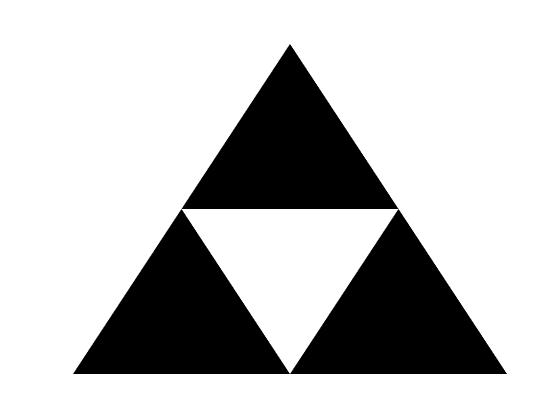}
		\includegraphics[width=\textwidth]{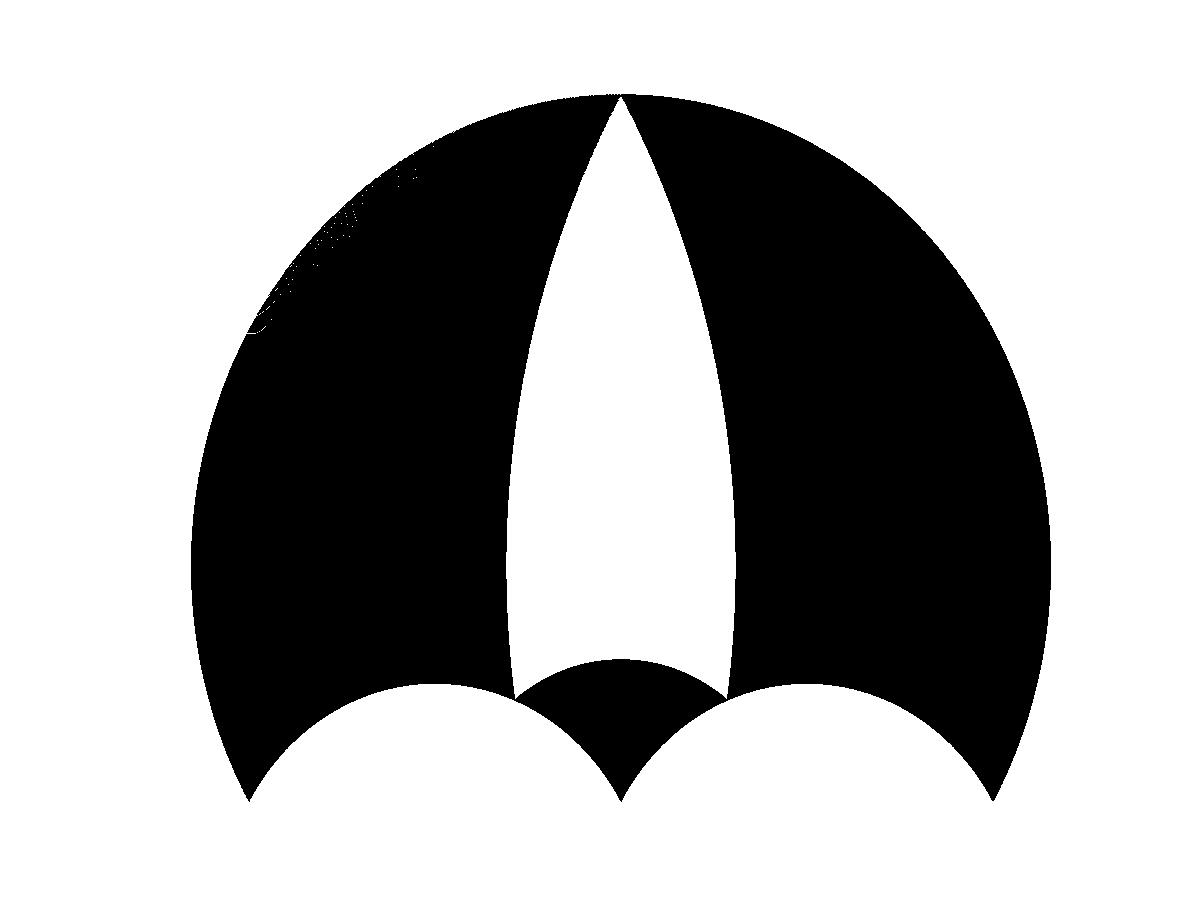}
		\includegraphics[width=\textwidth]{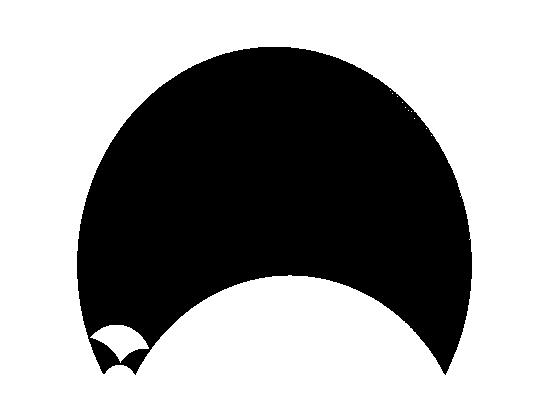}
		\caption{$\Phi O$}
	\end{subfigure}
	\hspace{-0.25cm}
	\begin{subfigure}{0.25\textwidth} 
		\includegraphics[width=\textwidth]{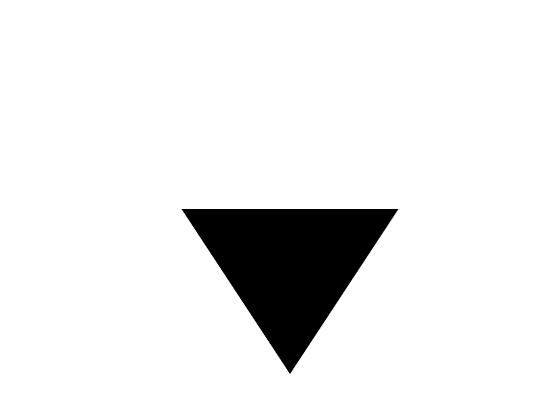}
		\includegraphics[width=\textwidth]{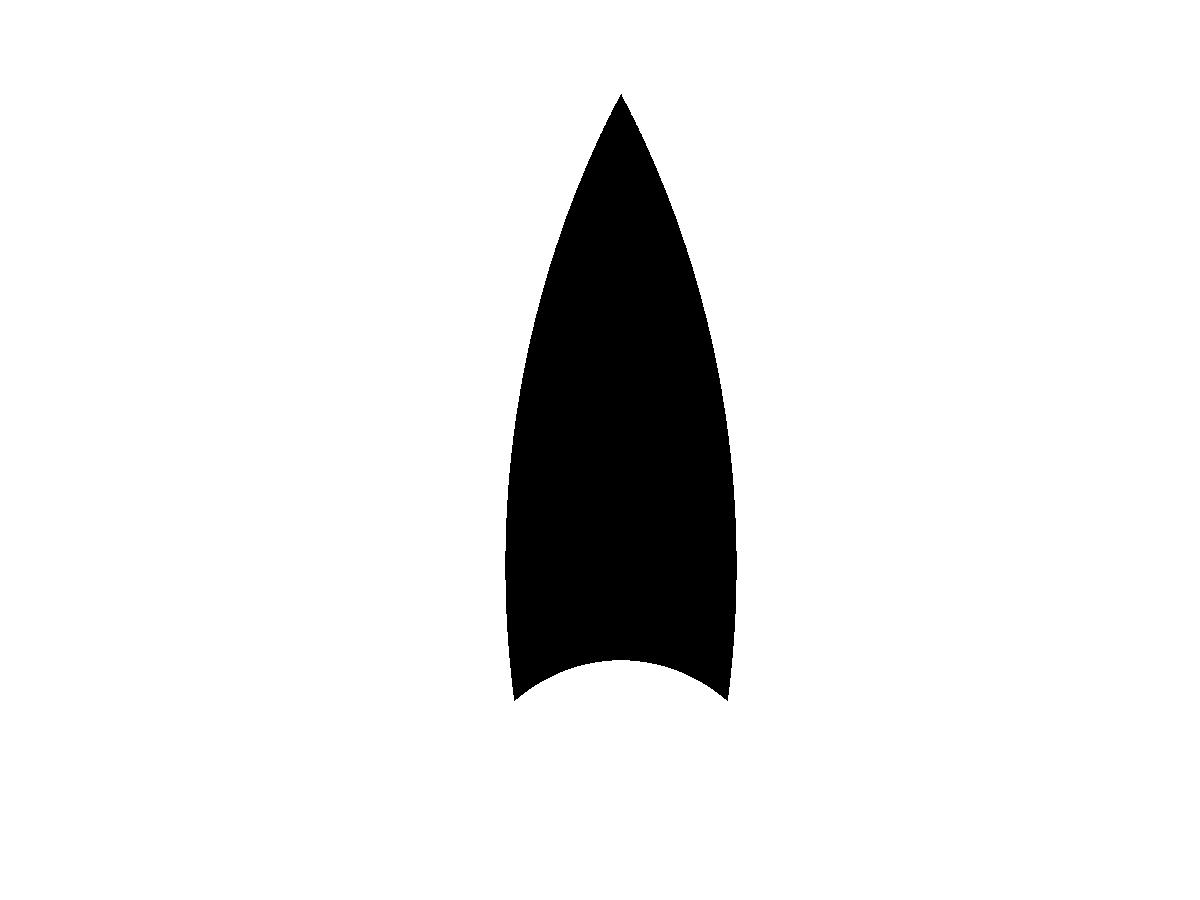}
		\includegraphics[width=\textwidth]{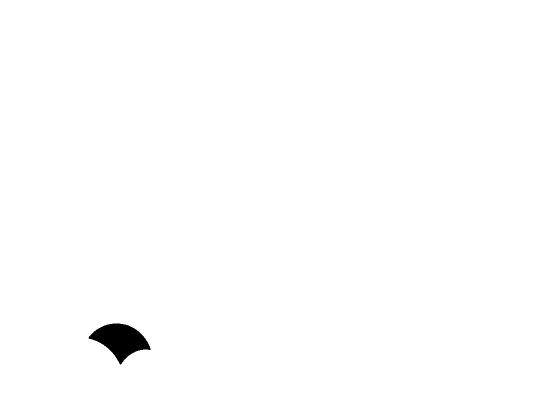}
		\caption{$\Gamma$}
	\end{subfigure}
	\caption{The self-conformal sets $F$ from Ex.~\ref{ex:pictures} (see Fig.~\ref{fig:fractals}), together with examples of associated feasible open sets $O$, $\Phi O$ and $\Gamma$.}
	\label{fig:opengamma}
\end{figure}
Next, we explain how to obtain the asymptotic behaviour of $\lambda_d(F_{\ee^{-t}})$ as $t\to\infty$. 
Assume that $\Phi$ satisfies the \emph{open set condition} (OSC), i\,e.\ there exists a \emph{feasible open set} $O\subset X$, which is non-empty and which satisfies $\phi_i O\subset O$ and $\phi_i O\cap \phi_j O=\emptyset$ for $i\neq j\in\Sigma$. 
Assume w.\,l.\,o.\,g.\ that $O$ is bounded.
For $\om=\om_1\cdots\om_n\in\Sigma^n$ write $\phi_{\om}\defeq\phi_{\om_1}\circ\cdots\circ\phi_{\om_n}$. 
Set $\Gamma\defeq O\setminus\bigcup_{i=1}^M \phi_i O$, $\Phi\Gamma\defeq\bigcup_{i=1}^M\phi_i\Gamma$ and note that
\begin{align}\label{eq:OUplus}
  O=\bigcup_{n=0}^{\infty}\bigcup_{u\in\Sigma^n}\phi_u\Gamma\cup\bigcap_{n=0}^{\infty}\Phi^n O,
\end{align}
where the unions are disjoint. 
For the sets from Fig.~\ref{fig:fractals} examples of sets $O$, with associated sets $\Phi O$ and $\Gamma$ are depicted in Fig.~\ref{fig:opengamma}.
We have $\Phi\left(\overline{\bigcap_{n=0}^{\infty}\Phi^n O}\right)=\overline{\bigcap_{n=0}^{\infty}\Phi^n O}$. Thus, $\overline{\bigcap_{n=0}^{\infty}\Phi^n O}$ is either empty or coincides with $F$ by uniqueness of the self-conformal set. Therefore, $\lambda_d\left(\bigcap_{n=0}^{\infty}\Phi^n O\right)\leq \lambda_d(F)$.
Let $D$ denote the Minkowski dimension of $F$. If $D<d$ then $ \lambda_d(F)=0$ and whence $\lambda_d\left(\bigcap_{n=0}^{\infty}\Phi^n O\right)=0$.
In the following we assume that $O$ can be chosen so that $\lambda_d(F_{\ee^{-t}}\cap\Gamma)\in\mathfrak{o}(\ee^{t(D-d)})$ with the little Landau symbol $\mathfrak o$. This is a mild condition, which is always satisfied for self-similar systems with any feasible open set $O$ (see \cite{Steffen_gamma}).   
Then \eqref{eq:OUplus} implies for $D<d$ that
\begin{align*}
  \lambda_d(F_{\ee^{-t}}\cap O)
  &= \sum_{n=0}^{\infty}\sum_{u\in\Sigma^n}\lambda_d(F_{\ee^{-t}}\cap\phi_u\Gamma)\\
  &=\sum_{\om\in\Sigma^m}\sum_{n=0}^{\infty}\sum_{u\in\Sigma^n}\lambda_d(F_{\ee^{-t}}\cap\phi_u\phi_{\om}\Gamma)+\mathfrak{o}(\ee^{t(D-d)})
\end{align*}
for any $m\in\mathbb N$. Further, suppose $\lambda_d(F_{\ee^{-t}}\cap\phi_u\phi_{\om}\Gamma)=\lambda_d((\phi_uF)_{\ee^{-t}}\cap\phi_u\phi_{\om}\Gamma)$. 
By \cite{Steffen_gamma}, for self-similar systems this assumption always holds true for some feasible open set $O$. (The open sets $O$ of Fig.~\ref{fig:opengamma} satisfy the above conditions, see Ex.~\ref{ex:pictures}.) If $m$ is large then $\phi_{\om}\Gamma$ is small, as each $\phi_i$ is strictly contracting. Since conformal maps locally behave like similarities and $\lambda_d$ is homogeneous of degree $d$, we can approximate $\lambda_d((\phi_uF)_{\ee^{-t}}\cap\phi_u\phi_{\om}\Gamma)$ by 
\begin{equation}\label{eq:approxMink}
\lvert\phi_u'(\pi\sigma\om x)\rvert^d\lambda_d(F_{\ee^{-t}/\lvert\phi_u'(\pi\sigma\om x)\rvert}\cap\phi_{\om}\Gamma)
\end{equation}
 with an arbitrary $x\in\Sigma^{\mathbb N}$.
 Here, $\sigma$ is the shift-map as before and $\pi\colon\Sigma^{\mathbb N}\to F$ is the \emph{code map} defined by $\{\pi(\om)\}\defeq\bigcap_{n=0}^{\infty}\phi_{\om\vert_n}(X)$, where $\om\vert_n\defeq\om_1\cdots\om_n$ is defined to be the \emph{subpath} of length $n$ of $\om$.
 In order to bring \eqref{eq:approxMink} into a form to apply the renewal theorem we define $\codefun\colon\Sigma^{\mathbb N}\to\mathbb R$ by
\[
\codefun(\om)\defeq-\log\lvert\phi'_{\om_1}(\pi\sigma\om)\rvert.
\]
The function $\codefun$ is called the \emph{geometric potential function} associated with the IFS $\Phi$. It carries important geometric information of $\Phi$ and $F$.
By definition $\exp(-S_n\codefun(u\om x))=\lvert\phi_u'(\pi\sigma\om x)\rvert$. Thus, $\lambda_d(F_{\ee^{-t}}\cap O)$ can be approximated by 
\[
\sum_{\om\in\Sigma^m}\sum_{n=0}^{\infty}\sum_{u\in\Sigma^n}\ee^{-dS_n\codefun(u\om x)}
\lambda_d(F_{\ee^{-t+S_n\codefun(u\om x)}}\cap\phi_{\om}\Gamma)+\mathfrak{o}(\ee^{t(D-d)}).
\]
Setting $f_y(t)\defeq\lambda_d(F_{\ee^{-t}}\cap\phi_{\om}\Gamma)$ for $y\in\Sigma^{\mathbb N}$, $\codefunc\defeq\mathds 1_{\Sigma^{\mathbb N}}$, $\eta\defeq-d\codefun$ and assuming the regularity conditions of Sec.~\ref{sec:RTneu}, we can apply the renewal theorem with dependent interarrival times (Thm.~\ref{thm:RT1}) and, if $\codefun$ is non-lattice, obtain
\begin{align*}
  &\sum_{n=0}^{\infty}\sum_{u\in\Sigma^n}\ee^{-dS_n\codefun(u\om x)}
\lambda_d(F_{\ee^{-t+S_n\codefun(u\om x)}}\cap\phi_{\om}\Gamma)
  =N(t,\om x)\\
  &\quad\sim \ee^{t\mdim}\frac{\eigenf_{-(d+\mdim)\codefun}(\om x)}{\int\codefun\textup{d}\mu_{-(d+\mdim)\codefun}}\int_{-\infty}^{\infty}\ee^{-T\mdim}\lambda_d(F_{\ee^{-T}}\cap\phi_{\om}\Gamma)\textup{d}T,
\end{align*}
where $\mdim>0$ is the unique value for which $P(-(d+\mdim)\codefun)=0$. Here, $P$ denotes the topological pressure function, see \eqref{eq:preddure} and Prop.~\ref{prop:pressureconvex}. It is proven in \cite{Bedford} that the Minkowski dimension $D$ of $F$ is the unique solution to $P(-D\codefun)=0$ thus, $d+\mdim=D$. Using approximation arguments (based on bounded distortion properties of $\Phi$) we all in all obtain, in the non-lattice situation, that
\begin{align*}
  \lambda_d(F_{\ee^{-t}}\cap O)
  \sim\ee^{t(D-d)}\hspace{-0.1cm}\lim_{m\to\infty}\hspace{-0.1cm}\sum_{\om\in\Sigma^m}\frac{\eigenf_{-D\codefun}(\om x)}{\int\codefun\textup{d}\mu_{-D\codefun}}\int_{-\infty}^{\infty}\hspace{-0.3cm}\ee^{-T(D-d)}\lambda_d(F_{\ee^{-T}}\cap\phi_{\om}\Gamma)\textup{d}T.
\end{align*}
This shows that $\codefun$ being non-lattice implies existence of the Minkowski content which can be determined explicitly.
The lattice and the average cases can be treated similarly.

In the special case that $F$ is self-similar, the approximation arguments are not needed and we can assume $m=0$. Further, $\eigenf_{-D\codefun}=\mathds 1_{\Sigma_A}$ and $\codefun$ is constant on cylinder sets of length one, taking the values $-\log(r_i)$ with $r_i\defeq\|\phi'_i\|_{\infty}$ and $\mu_{-D\codefun}([i])=r_i^D$, where $\|\cdot\|_{\infty}$ denotes the supremum-norm. Therefore, for self-similar sets $F$ we obtain
\begin{align*}
  \lambda_d(F_{\ee^{-t}}\cap O)
  \sim\ee^{t(D-d)}\frac{1}{-\sum_{i=1}^M\log(r_i)r_i^D}\int_{-\infty}^{\infty}\ee^{-T(D-d)}\lambda_d(F_{\ee^{-T}}\cap\Gamma)\textup{d}T
\end{align*}
if $\codefun$ is non-lattice.
If $\codefun$ is lattice then the respective formula is more involved but we can deduce from part \ref{it:averagethm} of the renewal theorem for dependent interarrival times (Thm.~\ref{thm:RT1}) that
\begin{align}\label{eq:latticefractal}
\lim_{t\to\infty}t^{-1}\int_0^t \ee^{-T(D-d)}\lambda_d(F_{\ee^{-t}}\cap O)\textup{d}T
= \frac{\int_{-\infty}^{\infty}\ee^{-T(D-d)}\lambda_d(F_{\ee^{-T}}\cap\Gamma)\textup{d}T}{-\sum_{i=1}^M\log(r_i)r_i^D}.
\end{align}

\begin{example}\label{ex:pictures}
  The following self-conformal sets are depicted in Figs.~\ref{fig:fractals} and \ref{fig:opengamma}. We will discuss the Minkowski content for the first set \subref{fig:Sierp_gasket} and restrict ourselves to the definitions and verification of the main conditions for \subref{fig:bat} and \subref{fig:moeb1}. In all the examples, $d=2$.
  \begin{enumerate}[label=(\alph*)]
    \item \label{it:Sierpinski} The Sierpinski gasket $F$ is the self-similar set associated with the IFS $\Phi\defeq\{\phi_1,\phi_2,\phi_3\colon\mathbb R^2\to\mathbb R^2\}$ given by $\phi_1(x)=x/2$, $\phi_2(x)=x/2+(1/2,0)$ and $\phi_3(x)=x/2+(1/4,\sqrt{3}/4)$. Its Minkowski dimension is $D=\log(3)/\log(2)$. The open triangle $O$ with vertices $(0,0)$, $(1,0)$, $(1/2,\sqrt{3}/2)$ is a feasible open set for $\Phi$. The set $\Gamma$ is the equilateral triangle with vertices $(1/2,0)$, $(1/4,\sqrt{3}/4)$, $(3/4,\sqrt{3}/4)$, and 
    \[
      \lambda_2(F_{\ee^{-t}}\cap\Gamma)
      =\begin{cases}
	\tfrac{3}{2}\ee^{-t}-3\sqrt{3}\ee^{-2t}& :t>-\log(\sqrt{3}/12)\\
	\tfrac{\sqrt{3}}{16}&  :t\leq-\log(\sqrt{3}/12).
      \end{cases}
    \]
    Thus, $\lambda_2(F_{\ee^{-t}}\cap\Gamma)\in\mathfrak{o}(\ee^{t(D-2)})$.
    We have $r_i\defeq\|\phi'_i\|_{\infty}=1/2$ for $i\in\{1,2,3\}$. Hence, the system is lattice and \eqref{eq:latticefractal} yields that 
    \begin{align*}
      &\lim_{t\to\infty} t^{-1}\int_0^t \ee^{-T(D-2)}\lambda_2(F_{\ee^{-T}}\cap O)\textup{d}T\\
      &\quad= \frac{\sqrt{3}^{-(D+1)}}{\log(2)}\left[\frac{1}{2-D}+\frac{2}{D-1}-\frac{1}{D}\right]
      \approx 1{.}8126.
    \end{align*}
    \item The set $F$ depicted in Fig.~\ref{fig:fractals}\subref{fig:bat} is the image of the Sierpinski gasket under the complex M\"obius transform $g_1\colon\mathbb{C}\to\mathbb C$, $z\mapsto-\im z/(\sqrt{3}z-\sqrt{3}-\im)$. An associated IFS is $\{g_1\circ\phi_i\circ g_1^{-1}\mid i\in\{1,2,3\}\}$ with $\phi_1,\phi_2,\phi_3$ as in \ref{it:Sierpinski}.  The maps $g_1\circ\phi_i\circ g_1^{-1}$ are contractive conformal $\mathcal{C}^{1+\alpha}$-diffeomorphisms, but not similitudes. Indeed, $F$ is self-conformal but not self-similar. It is known that the Minkowski dimension is stable under bi-Lipschitz maps, see \cite[Ch.~3.2]{Falconer_Foundation}. Thus, $D=\log(3)/\log(2)$. Since $\partial \Gamma\subseteq F$ is one-dimensional we have that $\lambda_2(F_{\ee^{-t}}\cap\Gamma)\in\mathfrak{O}(\ee^{-t})$ with the big Landau symbol $\mathfrak{O}$, whence $\lambda_2(F_{\ee^{-t}}\cap\Gamma)\in\mathfrak{o}(\ee^{t(D-d)})$.
    \item In Fig.~\ref{fig:fractals}\subref{fig:moeb1} the image of the Sierpinski gasket under the complex M\"obius transform $g_2\colon\mathbb C\to\mathbb C$, $z\mapsto(1-\sqrt{3}\im)z/((12+10\sqrt{3}\im)z-11-11\sqrt{3}\im)$ is depicted.
    Analogous to \subref{fig:bat} an associated IFS is given by $\{g_2\circ\phi_i\circ g_2^{-1}\mid i\in\{1,2,3\}\}$ and $\lambda_2(F_{\ee^{-t}}\cap\Gamma)\in\mathfrak{O}(\ee^{-t})$.
\end{enumerate}
\end{example}

A special case of the renewal theorems from the present article and of the above-stated geometric results is presented in the author's doctorate thesis \cite{Diss}. The setting of self-similar sets has been studied e.\,g.\ in \cite{Lapidus_Frankenhuysen_Springer,Falconer_Minkowski,Gatzouras,Steffen_gamma}. 
Self-conformal subsets of $\mathbb R$ and limit sets of graph-directed systems in $\mathbb R$, including Fuchsian groups of Schottky type, are treated in \cite{Marc1} and \cite{Marc2}, where results of \cite{Lalley} were applied. 
However, the setting of \cite{Lalley} is too restrictive for higher dimensional Euclidean spaces $\mathbb R^d$, making necessary the renewal theorems developed in the present article.

\subsection{Organisation of the article}
This article is organised as follows. 
In Sec.~\ref{sec:preliminaries} we introduce the relevant notions from Perron-Frobenius theory, which appear in the statements of the main theorems. 
Sec.~\ref{sec:results} is broken down into two parts. Sec.~\ref{sec:RTneu} is devoted to the presentation of the main results (Thms.~\ref{thm:RT1} and \ref{thm:RT2}) and in Sec.~\ref{sec:RTcorollaries} several corollaries are provided, demonstrating that the key renewal theorem for discrete measures, Lalley's renewal theorem and certain Markov renewal theorems can be recovered from our main theorems. 
Finally, in Sec.~\ref{sec:proofs} we provide the proofs of Thms.~\ref{thm:RT1} and \ref{thm:RT2}.

%% file: subshifts.tex
Here, we assemble preliminaries and fix notation. References for the exposition below are \cite{Bowen_equilibrium,Walters_Buch}.

\subsection{Subshifts of finite type -- Admissible paths of a random walk through $\Sigma$}\label{sec:subshifts}
We call $(\Sigma_A,\sigma)$ a \emph{subshift of finite type} or \emph{topological Markov chain}, since the transition rule only takes the present position into account and not the past.
If all entries of $A$ are ones, then $\Sigma_A=\Sigma^{\mathbb N}$ and $(\Sigma_A,\sigma)$ is called the \emph{full shift} on $M$ symbols. 
The set of \emph{admissible words of length $n\in\mathbb N$} is defined by
\begin{equation}
\Sigma_A^n\defeq\{\om\in\Sigma^{n}\mid A(\om_k,\om_{k+1})=1\ \text{for}\ k\leq n-1\}.
\end{equation}

If $\om$ has infinite length or length $m\geq n$ we define $\om\vert_n\defeq\om_1\cdots\om_n$ to be the subpath of length $n$. Further, $[\om]\defeq\{\omneu_1\omneu_2\cdots\in\Sigma_A\mid \omneu_i=\om_i\ \text{for}\ i\leq n\}$ is the \emph{$\om$-cylinder set} for $\om\in\Sigma_A^n$.

\subsection{Continuous and H\"older-continuous functions}\label{sec:cts}
Equip $\Sigma^{\mathbb N}$ with the product topology of the discrete topologies on $\Sigma$ and equip $\Sigma_A\subset\Sigma^{\mathbb N}$ with the subspace topology, i.\,e.\ the weakest topology with respect to which the canonical projections onto the coordinates are continuous. The spaces of continuous complex- and real-valued functions on $\Sigma_A$ are respectively denoted by $\mathcal C(\Sigma_A)$ and $\mathcal C(\Sigma_A,\mathbb R)$. Elements of $\mathcal C(\Sigma_A,\mathbb R)$ are called \emph{potential functions}.
\begin{definition}\label{defn:continuous}
  For $\codefun\in\mathcal C(\Sigma_A)$, $\alpha\in(0,1)$ and $n\in\mathbb N_0$ define 
  \begin{align*}
    \text{var}_n(\codefun)&\defeq\sup\{\lvert \codefun(\om)-\codefun(\omneu)\rvert\mid \om, \omneu\in\Sigma_A\ \text{and}\ \om_i=\omneu_i\ \text{for all}\ i\in\{1,\ldots,n\}\},\\
    \lvert \codefun\rvert_{\alpha}&\defeq\sup_{n\geq 0}\frac{\text{var}_n(\codefun)}{\alpha^{n}},\\
    \mathcal F_{\alpha}(\Sigma_A)&\defeq\{\codefun\in\mathcal C(\Sigma_A)\mid \lvert \codefun\rvert_{\alpha}<\infty\}\ \text{and}\  
    \mathcal F_{\alpha}(\Sigma_A,\mathbb R)\defeq \mathcal F_{\alpha}(\Sigma_A)\cap \mathcal C(\Sigma_A,\mathbb R).
  \end{align*}
Elements of $\mathcal F_{\alpha}(\Sigma_A)$ are called \emph{$\alpha$-H\"older continuous} functions on $\Sigma_A$. 
\end{definition}

\begin{definition}\label{defn:lattice}
  Functions $\codefun_1,\codefun_2\in\mathcal{C}(\Sigma_A)$ are called \emph{co-homologous}, if there exists $\psi\in\mathcal{C}(\Sigma_A)$ such that $\codefun_1-\codefun_2=\psi-\psi\circ\sigma$. A function $\codefun\in\mathcal{C}(\Sigma_A,\mathbb R)$ is said to be \emph{lattice}, if it is co-homologous to a function whose range is contained in a discrete subgroup of $\mathbb R$. Otherwise, we say that $\codefun$ is \emph{non-lattice}.
\end{definition}

\begin{remark}\label{rmk:strictlypos}
	$S_n\xi$ being strictly positive for some $n\in\mathbb N$ is equivalent to $\xi$ being co-homologous to a strictly positive function $\z$. To see this, note that $\z,\xi\in\mathcal C(\Sigma_A,\mathbb R)$ are co-homologous if and only if $S_m\z(x)=S_m\xi(x)$ for all $m\in\mathbb N$ and $x\in\Sigma_A$ with $\sigma^mx=x$. First, suppose that $S_n\xi$ is strictly positive. Let $\z\defeq S_n\xi/n$. Then $S_m\xi(x)=S_m\z(x)$ for all $x\in\Sigma_A$ with $\sigma^mx=x$. Second, suppose that $\xi=\z+\psi-\psi\circ\sigma$ for some $\psi\in\mathcal C(\Sigma_A,\mathbb R)$ and $\z$ satisfying $\z\geq\e>0$ for all $x\in\Sigma_A$. Then $S_m\xi(x)=S_m\z(x)+\psi(x)-\psi\circ\sigma^{m+1}(x)\geq m\e+\textup{var}_0(\psi)>0$ for sufficiently large $m\in\mathbb N$.
\end{remark}

\subsection{Topological pressure function and Gibbs measures}
  The \emph{topological pressure function} \linebreak $P\colon\mathcal C(\Sigma_A,\mathbb R)\to\mathbb R$ is given by the well-defined limit
  \begin{equation}\label{eq:preddure}
    P(\codefun)\defeq\lim_{n\to\infty}n^{-1}\log\sum_{\om\in\Sigma_A^n}\exp\sup_{\omneu\in[\om]}S_n\codefun(\omneu).
  \end{equation}
\begin{proposition}\label{prop:pressureconvex}
 Let $\codefun,\eta\in\mathcal C(\Sigma_A,\mathbb R)$ be so that $S_n\codefun$ is strictly positive on $\Sigma_A$, for some $n\in\mathbb N$. Then $s\mapsto P(\eta+s\codefun)$ is continuous, strictly monotonically increasing and convex with $\lim_{s\to -\infty}P(\eta+s\codefun)=-\infty$ and $\lim_{s\to\infty}P(\eta+s\codefun)=\infty$. Hence, there is a unique $\mdim\in\mathbb R$ for which $P(\eta-\mdim \codefun)=0$.
\end{proposition}

A finite Borel measure $\mu$ on $\Sigma_A$ is said to be a \emph{Gibbs measure} for $\codefun\in\mathcal C(\Sigma_A,\mathbb R)$ if there exists a constant $c>0$ such that 
\begin{equation}\label{eq:Gibbs}
	c^{-1}
	\leq \frac{\mu([\om\vert_n])}{\exp(S_n \codefun(\om)-n\cdot P(\codefun))}
	\leq c
\end{equation}
for every $\om\in\Sigma_A$ and $n\in\mathbb N$.

\subsection{Ruelle's Perron-Frobenius theorem}\label{sec:RuellesPF}
By \cite[Thm.~2{.}16, Cor.~2{.}17]{Walters_convergence} and \cite[Theorem 1{.}7]{Bowen_equilibrium}, for each $\codefun\in\mathcal F_{\alpha}(\Sigma_A,\mathbb R)$, some $\alpha\in(0,1)$, there exists a unique Borel probability measure $\nu_{\codefun}$ on $\Sigma_A$ satisfying $\mathcal L_{\codefun}^*\nu_{\codefun}=\eigenv_{\codefun}\nu_{\codefun}$ for some $\eigenv_{\codefun}>0$. This equation uniquely determines $\eigenv_{\codefun}$, which satisfies $\eigenv_{\codefun}=\exp(P(\codefun))$ and which coincides with the spectral radius of $\PF_{\codefun}$.
Further, there exists a unique strictly positive eigenfunction $\eigenf_{\codefun}\in\mathcal{C}(\Sigma_A,\mathbb R)$ satisfying $\mathcal L_{\codefun} h_{\codefun}=\eigenv_{\codefun} \eigenf_{\codefun}$ and $\int \eigenf_{\codefun}\textup{d}\nu_{\codefun}=1$.
Define $\mu_{\codefun}$ by $\textup{d}\mu_{\codefun}/\textup{d}\nu_{\codefun}=\eigenf_{\codefun}$. This is the unique $\sigma$-invariant Gibbs measure for the potential function $\codefun$. 
Additionally, under some normalisation assumptions we have convergence of the iterates of the Ruelle-Perron-Frobenius operator to the projection onto the one-dimensional subspace generated by its eigenfunction $\eigenf_{\codefun}$, namely
\begin{align}
\lim_{m\to\infty}\|\eigenv_{\codefun}^{-m}\mathcal L_{\codefun}^m\psi - \textstyle{\int}\psi\textup{d}\nu_{\codefun}\cdot \eigenf_{\codefun}\|_{\infty}=0\quad\forall\ \psi\in\mathcal C(\Sigma_A,\mathbb R).
\end{align}
Prop.~\ref{prop:pressureconvex} and the relation $\eigenv_{\codefun}=\exp(P(\codefun))$ imply the following.
\begin{proposition}\label{thm:eigenvalueone}
	Let $\codefun,\eta\in\mathcal C(\Sigma_A,\mathbb R)$ be such that for some $n\in\mathbb N$ the $n$-th Birkhoff sum $S_n\codefun$ of $\codefun$ is strictly positive on $\Sigma_A$. Then $s\mapsto\eigenv_{\eta+s\codefun}$ is continuous, strictly monotonically increasing, log-convex in $s\in\mathbb R$ with $\lim_{s\to -\infty}\eigenv_{\eta+s\codefun}=0$ and satisfies $\lim_{s\to\infty}\eigenv_{\eta+s\codefun}=\infty$. The unique $\mdim\in\mathbb R$ from Prop.~\ref{prop:pressureconvex} is the unique $\mdim\in\mathbb R$ for which $\eigenv_{\eta-\mdim\codefun}=1$. 
\end{proposition}

\begin{remark}
	In the proof of our renewal theorems, we will work with an analytic continuation of the pressure function to the complex domain and with complex Perron Frobenius theory. These complex quantities and their properties will be presented in Sec.~\ref{sec:PFanalyticp}.
\end{remark}

%% file: renewal.tex
We break this section into two parts. In Sec.~\ref{sec:RTneu} we give a precise statement of the renewal theorems (Thms.~\ref{thm:RT1} and \ref{thm:RT2}). In Sec.~\ref{sec:RTcorollaries}, we describe how the classical key renewal theorem for finitely supported measures, Lalley's renewal theorem for counting measures, and versions of Markov renewal theorems can be recovered from the results of Sec.~\ref{sec:RTneu}.

\subsection{Renewal theorems with dependent interarrival times}\label{sec:RTneu}
Here, we make our assumptions on the renewal function $N$ from \eqref{eq:intro:renfcn} more precise. 
We fix $x\in\Sigma_A$ and take $\alpha\in(0,1)$.
Further, $\codefun,\eta\in\mathcal{F}_{\alpha}(\Sigma_A,\mathbb R)$ are so that $S_n\codefun$ is strictly positive on $\Sigma_A$ for some $n\in\mathbb N$, see also Rem.~\ref{rmk:strictlypos}. Henceforth, we let $\mdim>0$ denote the unique real for which $\eigenv_{\eta-\mdim\codefun}=1$ (see Prop.~\ref{thm:eigenvalueone}). For $y\in\Sigma_A$ and $t\in\mathbb R$ we write 
\[
\widetilde{f}_y(t)=\chi(t)\cdot f_y(t)
\]
with non-negative but not identically zero $\codefunc\in\mathcal{F}_{\alpha}(\Sigma_A,\mathbb R)$, where $\renfcn_y\colon\mathbb R\to\mathbb R$, for $y\in\Sigma_A$, needs to satisfy some regularity conditions. 
One of these requires equi directly Riemann integrability, a condition which is motivated by an assumption of the classical key renewal theorem (see Sec.~\ref{sec:keyRT}).

\begin{definition}\label{def:dRi}
   For a function $f\colon\mathbb R\to\mathbb R$, $h>0$ and $k\in\mathbb Z$ set
   \begin{align*}
     \underline{m}_k(f,h)&\defeq\inf\{f(t)\mid (k-1)h\leq t<kh\}\qquad\text{and}\\
     \overline{m}_k(f,h)&\defeq\sup\{f(t)\mid (k-1)h\leq t<kh\}.
   \end{align*}
   The function $f$ is called \emph{directly Riemann integrable (d.\,R.\,i.)} if for some sufficiently small $h>0$
   \begin{align*}
     \underline{R}(f,h)\defeq\sum_{k\in\mathbb Z}h\cdot\underline{m}_k(f,h)\qquad\text{and}\qquad 
     \overline{R}(f,h)\defeq\sum_{k\in\mathbb Z}h\cdot\overline{m}_k(f,h)
   \end{align*}
   are finite and tend to the same limit as $h\to 0$.
   We call a family of functions $\{f_x\colon\mathbb R\to\mathbb R \mid x\in I\}$ with some index set $I$ \emph{equi directly Riemann integrable (equi d.\,R.\,i.)} if for some sufficiently small $h>0$
   \begin{align*}
     \underline{R}(h)\defeq\sum_{k\in\mathbb Z}h\cdot\inf_{x\in I}\underline{m}_k(f_x,h)\qquad\text{and}\qquad 
     \overline{R}(f,h)\defeq\sum_{k\in\mathbb Z}h\cdot\sup_{x\in I}\overline{m}_k(f_x,h)
   \end{align*}
   are finite and tend to the same limit as $h\to 0$.
\end{definition}
D.\,R.\,i.\ excludes wild oscillations of the function at infinity and is stronger than Riemann integrability. For further insights into this notion we refer the reader to \cite[Ch.~XI]{Feller} and \cite[Ch.~B.V]{Asmussen}.\\

The \emph{regularity conditions} are as follows. 

\begin{enumerate}[label=(\Alph*)]
  \item\label{it:Lebesgueintegrable} \emph{Lebesgue integrability.} For any $x\in\Sigma_A$ the Lebesgue-integral
  \[
  \int_{-\infty}^{\infty}\ee^{-t\mdim}\lvert\renfcn_x(t)\rvert\textup{d}t
  \]      
  exists.
  \item\label{it:boundedC} \emph{Boundedness of $N$.} There exists $\mathfrak C>0$ such that $\ee^{-t\mdim}N^{\text{abs}}(t,x)\leq\mathfrak C$ for all $x\in\Sigma_A$ and $t\in\mathbb R$, where
  \[
  N^{\text{abs}}(t,x)\defeq\sum_{n=0}^{\infty}\sum_{y:\sigma^n y=x}\codefunc(y)\lvert\renfcn_y(t-S_n\codefun(y))\rvert\ee^{S_n\eta(y)}.
  \]
  \item\label{it:boundedneg} \emph{Exponential decay of $N$ on the negative half-axis.} There exist $\tilde{\mathfrak C}>0, s>0$ and $t^*\in\mathbb R$ such that $\ee^{-t\mdim}N^{\text{abs}}(t,x)\leq \tilde{\mathfrak C}\ee^{st}$ for all $t\leq t^*$.
\end{enumerate}
In the non-lattice situation we additionally need the following.
\begin{enumerate}[label=(\Alph*)]\setcounter{enumi}{3}
  \item\label{it:regular} \emph{Non-oscillatory.} The renewal function $N$ does not oscillate wildly at infinity, in particular at least one of the following is satisfied.
  \begin{enumerate}
  \item\label{it:monotonic} The function $t\mapsto \renfcn_x(t)$ is monotonic for any $x\in\Sigma_A$.
  \item\label{it:nldRi} The family $\{t\mapsto\ee^{-t\mdim}\lvert\renfcn_x(t)\rvert\mid x\in\Sigma_A\}$ is equi d.\,R.\,i.
  \item\label{it:nonoscill} The oscillation of $N$ is small in the sense that
  \begin{align*}
        &\liminf_{\e\to 0}\limsup_{r\to\infty}\sup_{\tilde{\e}\in[0,\e]}\ee^{-(r-\tilde{\e})\mdim}N(r-\tilde{\e},x)\\
        &\qquad=\limsup_{\e\to 0}\liminf_{r\to\infty}\inf_{\tilde{\e}\in[0,\e]}\ee^{-(r-\tilde{\e})\mdim}N(r-\tilde{\e},x).
  \end{align*}
  \end{enumerate}
\end{enumerate}
Note that Conditions \ref{it:Lebesgueintegrable} to \ref{it:regular} are a weakening of imposed assumptions of known renewal theorems. 
In Thm.~\ref{thm:RT2} it is for instance shown that equi d.\,R.\,i.\ of $\{t\mapsto\ee^{-t\mdim}\lvert \renfcn_x(t)\rvert\mid x\in\Sigma_A\}$ and exponential decay of $t\mapsto \ee^{-t\mdim} \renfcn_x(t)$ on the negative half-axis imply  \ref{it:Lebesgueintegrable} to \ref{it:regular}.
The conditions \ref{it:Lebesgueintegrable} to \ref{it:boundedneg} and \ref{it:monotonic} are motivated by \cite{Lalley}.

For $t\in\mathbb R$ we write $\lfloor t\rfloor$ for 
the largest integer $k\in\mathbb Z$ satisfying $k\leq t$. Moreover, we set $\{t\}\defeq t-\lfloor t\rfloor\in[0,1)$. Notice, for $t\in\mathbb R$ positive, $\lfloor t\rfloor$ is the integer part and $\{t\}$ is the fractional part of $t$.

\begin{theorem}[Renewal theorem]\label{thm:RT1}
   Assume that $x\mapsto \renfcn_x(t)$ is $\alpha$-H\"older continuous for any $t\in\mathbb R$ and that Conditions \ref{it:Lebesgueintegrable} to \ref{it:boundedneg} hold.
   \begin{enumerate}
   \item\label{it:RT1:nl} If $\codefun$ is non-lattice and \ref{it:regular} is satisfied, then 
   \begin{equation*}
    N(t,x)\sim\ee^{t\mdim}\underbrace{
\frac{\eigenf_{\eta-\mdim \codefun}(x)}{\int \codefun\textup{d}\mu_{\eta-\mdim \codefun}}\int_{\Sigma_A} \codefunc(y)\int_{-\infty}^{\infty}\ee^{-T\mdim}\renfcn_y(T)\textup{d}T\textup{d}\nu_{\eta-\mdim \codefun}(y)
}_{\eqdef G(x)}
   \end{equation*}
   as $t\to\infty$, uniformly for $x\in\Sigma_A$.
   \item\label{it:RT1:l} Assume that $\codefun$ is lattice and let $\z,\psi\in\mathcal{C}(\Sigma_A,\mathbb R)$ satisfy the relation
   \[
   \codefun-\z=\psi-\psi\circ\sigma,
   \]
   where $\z(\Sigma_A)\subseteq\aaa\mathbb Z$ for some $\aaa>0$. Suppose that $\codefun$ is not co-homologous to any function with values in a proper subgroup of $\aaa\mathbb Z$. Then
  \begin{align*}
    N(t,x)
    \sim \ee^{t\mdim}\tilde{G}_x(t)
  \end{align*}
  as $t\to\infty$, uniformly for $x\in\Sigma_A$. Here $\tilde{G}_x$ is periodic with period $\aaa$ and
  \begin{align*}
    \tilde{G}_x(t)
    &\defeq \int_{\Sigma_A}\codefunc(y)
\sum_{l=-\infty}^{\infty}\ee^{-\aaa l\mdim}\renfcn_y\left(\aaa l+\aaa\left\{\tfrac{t+\psi(x)}{\aaa}\right\}-\psi(y)\right)
    \textup{d}\nu_{\eta-\mdim\z}(y)\\
    &\qquad\times \ee^{-\aaa\big{\{}\frac{t+\psi(x)}{\aaa}\big{\}}\mdim}
    \frac{\aaa\ee^{\mdim\psi(x)}}{\int\z\textup{d}\mu_{\eta-\mdim\z}}\cdot\eigenf_{\eta-\mdim\z}(x).
  \end{align*}
  \item\label{it:RT1:av}We always have 
  \begin{equation*}
  \lim_{t\to\infty}t^{-1}\int_0^{t}\ee^{-T\mdim}N(T,x)\textup{d}T=G(x).
  \end{equation*}
  \end{enumerate}
\end{theorem}

\begin{theorem}\label{thm:RT2}
   Assume that $x\mapsto \renfcn_x(t)$ is $\alpha$-H\"older continuous for any $t\in\mathbb R$ and that the family $\{t\mapsto\ee^{-t\mdim}\lvert\renfcn_x(t)\rvert\mid x\in\Sigma_A\}$ is equi d.\,R.\,i. If there exist $\mathfrak{C}',s>0$ such that $\ee^{-t\mdim}\lvert\renfcn_x(t)\rvert\leq\mathfrak{C}'\ee^{st}$ for $t<0$ and $x\in\Sigma_A$, then \ref{it:Lebesgueintegrable} to \ref{it:regular} are satisfied and thus the statements of Thm.~\ref{thm:RT1} all hold.
\end{theorem}

\subsection[Corollaries to the renewal theorems]{Corollaries to Thms.~\ref{thm:RT1} and \ref{thm:RT2}}
\label{sec:RTcorollaries}

In this section we will see how established renewal theorems can be obtained as corollaries to Thms.~\ref{thm:RT1} and \ref{thm:RT2}. 

\subsubsection{The key renewal theorem for finitely supported measures}\label{sec:keyRT}

Thms.~\ref{thm:RT1} and \ref{thm:RT2} deal with renewal functions $N\colon\mathbb R\times\Sigma_A\to\mathbb R$. The special case of Thm.~\ref{thm:RT2} that $N$ is independent of $\Sigma_A$ gives the classical key renewal theorem: 

$N$ being independent of $\Sigma_A$ can be achieved by the following assumptions. 
First, $\Sigma_A=\Sigma^{\mathbb N}$ (i.\,e.\ full shift).
Second, $\renfcn_x = \renfcn$ is independent of $x\in\Sigma^{\mathbb N}$ implying that $z\colon\mathbb R\to\mathbb R$ with $z(t)\defeq\ee^{-\mdim t}\renfcn(t)$ is d.\,R.\,i. 
Third, $\codefunc=\mathds 1_{\Sigma_A}$.
Fourth and most importantly, $\codefun$ and $\eta$ are constant on cylinder sets of length one.
To emphasise local constancy, write 
$s_{\omneu}\defeq S_n\codefun(\omneu_1\cdots\omneu_n\om)$ and
$p_{\omneu}\defeq\exp\left[S_n(\eta-\mdim\codefun)(\omneu_1\cdots\omneu_n\om)\right]$
for $\omneu=\omneu_1\cdots\omneu_n\in\Sigma^n$ and $\om=\om_1\om_2\cdots\in\Sigma^{\mathbb N}$. 
Setting $Z(t)\defeq\ee^{-\mdim t}N(t)$
we obtain that 
\begin{equation}\label{eq:reneqFeller}
  Z(t)=\sum_{n=0}^{\infty}\sum_{\om\in\Sigma^n}z(t-s_{\om})p_{\om} \quad\text{and}\quad 
  Z(t)
  =\sum_{i=1}^M Z(t - s_i)p_i+ z(t),
\end{equation}
for $t\in\mathbb R$. The latter equation of \eqref{eq:reneqFeller} is the \emph{classical renewal equation}. It is equivalent to $Z=Z\star F+z$, where $F$ is the distribution which assigns mass $p_i$ to $s_i$ and where $\star$ denotes the convolution operator.
The assumption $S_n\codefun>0$ for some $n\in\mathbb N$ implies $s_i>0$ for all $i\in\Sigma$. Thus, $F$ is concentrated on $(0,\infty)$. 
On the other hand, any vector $(s_1,\ldots, s_M)$ with $s_1,\ldots,s_M>0$ determines a strictly positive function $\codefun\in\mathcal F_{\alpha}(\Sigma^{\mathbb N},\mathbb R)$ via $\codefun(\om_1\om_2\cdots)\defeq s_{\om_1}$.
Furthermore, in the setting of Thm.~\ref{thm:RT2}, $(p_1,\ldots,p_M)$ is a probability vector with $p_i\in(0,1)$ as
\[
  0
  =P(\eta-\mdim\codefun)
  =\lim_{n\to\infty}n^{-1}\log\bigg{(}\sum_{i\in\Sigma}p_i\bigg{)}^n
  =\log\sum_{i\in\Sigma}p_i
\]
by Prop.~\ref{prop:pressureconvex}. Thus, $F$ is a probability distribution.
On the other hand, any probability vector $(p_1,\ldots,p_M)$ with $p_1,\ldots,p_M\in(0,1)$ determines $\eta\in\mathcal F_{\alpha}(\Sigma^{\mathbb N},\mathbb R)$ via $\eta(\om_1\om_2\cdots)\defeq\log(p_{\om_1}\ee^{\mdim s_{\om_1}})$.

Consequently, Thm.~\ref{thm:RT2}  provides the asymptotic behaviour of $Z$ under the assumptions that $(p_1,\ldots,p_M)$ is a probability vector and that $s_1,\ldots,s_M>0$. In order to present the asymptotic terms in a common form, observe that $  \PF_{\eta-\mdim\codefun}\mathbf{1}=\mathbf{1}(x)$
for any $x\in\Sigma^{\mathbb N}$, where $\mathbf{1}=\mathds{1}_{\Sigma^{\mathbb N}}$. Thus, 
\begin{equation*}
        \eigenf_{\eta-\mdim\codefun}=\mathbf{1}
        \qquad\text{and}
        \qquad \mu_{\eta-\mdim\codefun}([i])=\nu_{\eta-\mdim\codefun}([i])= p_i,
\end{equation*}
where the last equality follows by considering the dual operator of $\PF_{\eta-\mdim\codefun}$.
If $\codefun$ is lattice then the range of $\codefun$ itself lies in a discrete subgroup of $\mathbb R$: If there exist $\z,\psi\in\mathcal C(\Sigma^{\mathbb N},\mathbb R)$ with $\codefun-\z=\psi-\psi\circ\sigma$ and $\z(\Sigma^{\mathbb N})\subset\aaa\mathbb Z$ for some $\aaa>0$, then $\codefun$ and $\z$ need to coincide on  $\{\om\in\Sigma^{\mathbb N}\mid\om=\sigma\om\}$. As every cylinder set of length one contains a periodic word of period one the claim follows. Hence, we can choose $\z=\codefun$ and $\psi$ to be the constant zero-function. 
We deduced:

\begin{corollary}[Key renewal theorem]\label{cor:Feller}
   Let $(p_1,\ldots, p_M)$ be a probability vector with $p_i\in(0,1)$ and $s_i>0$ for $i\in\{1,\ldots,M\}$, $M\geq 2$. Denote by $z\colon\mathbb R\to\mathbb R$ a d.\,R.\,i.\ function with $z(t)\leq \mathfrak{C}'\ee^{st}$ for $t<0$, some $\mathfrak{C}',s>0$. Let $Z\colon\mathbb R\to\mathbb R$ be as in \eqref{eq:reneqFeller}.
Then the following hold:
\begin{enumerate}
        \item If  $\{s_1,\ldots,s_M\}$ does not lie in a discrete subgroup of $\mathbb R$, then as $t\to \infty$
        \[
        Z(t)\sim \frac{1}{\sum_{i=1}^M p_i s_i}\int_{-\infty}^{\infty} z(T)\textup{d}T.
        \]
        \item If $\{s_1,\ldots, s_M\}\subset \aaa\cdot\mathbb Z$ and $\aaa>0$ is maximal, then as $t\to \infty$
        \[
        Z(t)\sim \frac{\aaa}{\sum_{i=1}^M p_i s_i}\sum_{l=-\infty}^{\infty} z(\aaa l+t).
        \]
        \item We always have 
        \[
        \lim_{t\to\infty}t^{-1}\int_0^tZ(T)\textup{d}T
        = \frac{1}{\sum_{i=1}^M p_i s_i}\int_{-\infty}^{\infty} z(T)\textup{d}T.
        \]      
\end{enumerate}
\end{corollary}

\begin{remark}
    In \cite[Ch.~XI]{Feller} the key renewal theorem is stated for functions $z$ which vanish on the negative half-axis. However, note in the above corollary exponential decay of $z$ on the negative half axis is allowed.
\end{remark}

\subsubsection{A renewal theorem for counting measures in symbolic dynamics}\label{sec:corlalley}

In \cite{Lalley} the case that $\eta$ is the constant zero-function in conjunction with $f_x\defeq\mathds{1}_{[0,\infty)}$ for every $x\in\Sigma_A$ is addressed. 
With these restrictions, Conditions
\ref{it:Lebesgueintegrable} and \ref{it:monotonic} are immediate and \ref{it:boundedC} as well as \ref{it:boundedneg} are shown in \cite[Lemma 8.1]{Lalley}. 
The renewal function from \eqref{eq:intro:renfcn} becomes
\begin{equation*}
\widetilde{N}(t,x)\defeq \sum_{n=0}^{\infty}\sum_{y:\sigma^ny=x}\codefunc(y)\mathds 1_{[0,\infty)}(t-S_n\codefun(y)),
\end{equation*}
which is a counting function. 
Thm.~\ref{thm:RT1} provides its asymptotic behaviour as $t\to\infty$, recovering \cite[Thms.~1 to 3]{Lalley} and yielding  an exact asymptotic in the lattice situtation. Since this is not explicitely given in \cite{Lalley}, we limit ourselves to the lattice case in the following corollary.

\begin{corollary}\label{cor:Lalley}
  For a fixed $\alpha\in(0,1)$, let $\codefun,\codefunc\in\mathcal F_{\alpha}(\Sigma_A,\mathbb R)$ be such that $\codefunc$ is non-negative but not identically zero and that there exists an $n\in\mathbb N$ for which $S_n\codefun$ is strictly positive. 
  Assume that $\codefun$ is lattice and let $\z,\psi\in\mathcal{C}(\Sigma_A)$ denote functions which satisfy $\codefun-\z=\psi-\psi\circ\sigma$,
  where $\z(\Sigma_A)\subseteq\aaa\mathbb Z$ for some $\aaa>0$. Suppose that $\codefun$ is not co-homologous to a function whose range lies in a proper subgroup of $\aaa\mathbb Z$. Uniformly for all $x\in\Sigma_A$ we have, as $t\to\infty$, that
  \begin{equation*}
    \widetilde{N}(t,x)\sim\frac{\aaa\eigenf_{-\mdim\z}(x)\int \codefunc(y)\ee^{\mdim\aaa\left\lfloor\frac{t}{\aaa}-\frac{\psi(y)-\psi(x)}{\aaa}\right\rfloor}\textup{d}\nu_{-\mdim\z}(y)}{(1-\ee^{-\mdim\aaa})\int\z\textup{d}\mu_{-\mdim\z}}.
  \end{equation*}
\end{corollary}

\subsubsection{A Markov renewal theorem}\label{sec:corMarkov}

If we assume that $\eta$ and $\codefun$ are constant on cylinder sets of length two, then the point process with interarrival times $W_0,W_1,\ldots$ becomes a Markov random walk: 
To see this, define $\tilde{\eta},\tilde{\codefun}\colon\Sigma_A^2\to\mathbb R$ by $\tilde{\eta}(ij)\defeq\eta(ij\om)$ and $\tilde{\codefun}(ij)\defeq\codefun(ij\om)$ for any $\om\in\Sigma_A$ for which $ij\om\in\Sigma_A$. Then
\begin{align*}
   \mathbb P(X_1=i\mid X_0X_{-1}\cdots=x)
   = \ee^{\eta(ix)}
   &= \ee^{\tilde{\eta}(ix_1)}
   =\mathbb P(X_1=i\mid X_0=x_1).
\end{align*}
Thus, $(X_n)_{n\in\mathbb Z}$ is a Markov chain.
Further, $W_n=\codefun(X_{n+1}X_nX_{n-1}\cdots) =\tilde{\codefun}(X_{n+1}X_n)$
implies that the interarrival times $W_0,W_1,\ldots$ are Markov dependent on $(X_n)_{n\in\mathbb Z}$.
Applying Thms.~\ref{thm:RT1} and \ref{thm:RT2} to such Markov random walks gives a \emph{Markov renewal theorem}. In order to state it in a common form we present several simplifications and conversions in the following.

Under the current assumptions, the analogue of the transition kernel $U$ from the introduction becomes a transition kernel $\tilde{U}\colon\Sigma\times(\mathcal{P}(\Sigma)\otimes\mathcal B(\mathbb R))\to\mathbb R$,
    \begin{align*}
    \tilde{U}(i,\{j\}\times(-\infty,t])
    \defeq&\ \mathbb P(X_{n+1}=j,W_n\leq t\mid X_n=i)\\
    =&\ \begin{cases}
    \mathds 1_{(-\infty,t]}(\tilde{\codefun}(ji))\ee^{\tilde{\eta}(ji)} &:\ ji\in\Sigma_A^2\\
    0 &:\ \text{otherwise}.
    \end{cases}
    \end{align*}
Set $\tilde{F}_{ij}(t)\defeq\tilde{U}(i,\{j\}\times(-\infty,t])$ and define $F\defeq(\tilde{F}_{ij})_{i,j\in\Sigma}$ to be the matrix with entries $F_{ij}\defeq\|\tilde{F}_{ij}\|_{\infty}=\exp(\tilde{\eta}(ji))\mathds 1_{\Sigma_A^2}(ji)$. Then, $F$ is primitive if and only if $A$ is primitive.
Moreover, $\tilde{F}_{ij}$ is a distribution function of a discrete measure. A distribution function is called \emph{lattice} if its set of discontinuities lies in a discrete subgroup of $\mathbb R$. Otherwise, it is called \emph{non-lattice}. Thus, $\tilde{F}_{ij}$ is lattice if and only if $\codefun$ is lattice.
For $s\in\mathbb R$  and $i,j\in \Sigma$ we have 
   \begin{align*}
   B_{ij}(s)
   &\defeq\int\ee^{s T}\tilde{F}_{ij}(\textup{d}T)
   =\begin{cases}
   \exp(\tilde{\eta}(ji)+s\tilde{\codefun}(ji))&:\ ji\in\Sigma_A^2\\
   0&:\ \text{otherwise}.
   \end{cases}
   \end{align*}
Setting $B(s)\defeq(B_{ij}(s))_{i,j\in\Sigma}$ we see that the action of $B(s)$ on vectors coincides with the action of the Ruelle-Perron-Frobenius operator $\PF_{\eta+s\codefun}$ on functions $g\colon\Sigma_A\to\mathbb R$ which are constant on cylinder sets of length one. That is, setting $\tilde{g}_i\defeq g(ix)$, for $x\in\Sigma_A$ with $ix\in\Sigma_A$, gives 
   \[
   \PF_{\eta+s\codefun}g(ix)
   =\sum_{j\in\Sigma,\,ji\in\Sigma_A^2}\ee^{\tilde{\eta}(ji)+s\tilde{\codefun}(ji)}\tilde{g}_j
   =\sum_{j\in\Sigma} B_{ij}(s)\tilde{g}_j
   =(B(s)\tilde{g})_i.
   \]
   By the Perron-Frobenius theorem for matrices there is a unique $s$ for which $B(s)$ has spectral radius one. By the above this value coincides with the unique $s$ for which $\PF_{\eta+s\codefun}$ has spectral radius one, which we denoted by $\mdim$ in Prop.~\ref{thm:eigenvalueone}. Similarly, $\eigenf_{\eta-\mdim\codefun}$ is constant on cylinder sets of length one. Thus, setting $\eigenf_i\defeq\eigenf_{\eta-\mdim\codefun}(ix)$ for $x\in\Sigma_A$ with $ix\in\Sigma_A$ we obtain a vector $(\eigenf_i)_{i\in\Sigma}$ with strictly positive entries which satisfies $B(-\mdim)\eigenf=\eigenf$, since
      \[
      (B(-\mdim)\eigenf)_i
      =\PF_{\eta-\mdim\codefun}\eigenf_{\eta-\mdim\codefun}(ix)
      =\eigenf_{\eta-\mdim\codefun}(ix)
      =\eigenf_i.
      \]
Moreover, the vector $\nu$ given by $\nu_i\defeq\nu_{\eta-\mdim\codefun}([i])$ satisfies $\nu_i>0$ for all $i\in\Sigma$ and $\nu B(-\mdim)=\nu$, since $\PF^{*}_{\eta-\mdim\codefun}\nu_{\eta-\mdim\codefun}=\nu_{\eta-\mdim\codefun}$. By the Perron-Frobenius theorem $h$ and $\nu$ are unique with these properties. 
Additionally assuming $\codefunc=\mathds 1_{\Sigma_A}$ and that $\renfcn_x$ only depends on the first letter of $x\in\Sigma_A$ it follows that $N(t,x)$ only depends on the first letter of $x$. Thus, for $i\in\Sigma$ write $N(t,i)\defeq N(t,ix)$ with $x\in\Sigma_A$ for which $ix\in\Sigma_A$. Now, the renewal equation becomes
\begin{align}
	\begin{aligned}\label{eq:Markovrenewal}
        N(t,i)
        &=\sum_{j\in\Sigma,\ ji\in\Sigma_A^2} N(t-\tilde{\codefun}(ji),j)\ee^{\tilde{\eta}(ji)} + \renfcn_i(t)\\
        &=\sum_{j\in\Sigma} \int_{-\infty}^{\infty}N(t- u,j)\tilde{F}_{ij}(\textup{d}u)+ \renfcn_i(t),
        \end{aligned}
\end{align}
for $i\in\Sigma$, where $\renfcn_i(t)\defeq \renfcn_{ix}(t)$ for $x\in\Sigma_A$ with $ix\in\Sigma_A$. The system of equations given in \eqref{eq:Markovrenewal} for varying $i\in\Sigma$ is called a \emph{Markov renewal equation}, \emph{multivariate renewal equation} or \emph{system of coupled renewal equations}.

\begin{corollary}[A Markov renewal theorem]
   Let $M\geq 2$ be an integer. For $i\in\{1,\ldots,M\}$ let $\renfcn_i\colon\mathbb R\to\mathbb R$ denote d.\,R.\,i.\ functions. Suppose that there exist $\mathfrak C',s>0$ such that $\ee^{-\mdim t}\lvert \renfcn_i(t)\rvert\leq\mathfrak C'\ee^{st}$ for $t<0$ and $i\in\{1,\ldots,M\}$. Let $\tilde{F}_{ij}(t)$ be as above and suppose that $F\defeq(\|\tilde{F}_{ij}\|_{\infty})_{i,j\in\Sigma}$ is primitive. Let $\mdim>0$ denote the unique positive real number for which the matrix $B(-\mdim)$ given by $B_{ij}(-\mdim)\defeq\int\ee^{-\mdim u}\tilde{F}_{ij}(\textup{d}u)$ has spectral radius one. Choose vectors $\nu$, $\eigenf$ with $\nu B(-\mdim)=\nu$, $B(-\mdim)\eigenf=\eigenf$ and $\nu_i,\eigenf_i>0$ for $i\in\Sigma$. Let $N(t,i)$ for $i\in\Sigma$ solve the Markov renewal equation \eqref{eq:Markovrenewal}.
   \begin{enumerate}
   \item If $\tilde{F}_{ij}$ is non-lattice, then
   \[
   \ee^{-\mdim t}N(t,i)
   \sim \frac{\eigenf_i\sum_{j=1}^M\nu_j\int\ee^{-\mdim T}\renfcn_j(T)\textup{d}T}{\sum_{k,j=1}^M\nu_k\eigenf_j\int T\ee^{-\mdim T}F_{kj}(\textup{d}T)}
   \eqdef G(i).
   \]
   \item We always have
   \[
   \lim_{t\to\infty} t^{-1}\int_{0}^t\ee^{-T\mdim}N(T,i)\textup{d}T=G(i).
   \]
   \end{enumerate}
\end{corollary}
A statement for the lattice situation can likewise be deduced from Thm.~\ref{thm:RT2}.
\begin{remark}
        The above theorem is presented in a similar form in \cite[Thm.~4{.}6]{Asmussen}. There, the matrix $F$ is required to be irreducible instead of primitive and the function $\tilde{\codefun}$, which $\tilde{F}_{ij}$ depends on, can be a random variable. On the other hand, each $\renfcn_i$ is required to be zero on the half-line $(-\infty,0)$ whereas here, we allow exponential decay on the negative axis.  More general versions of Markov renewal theorems can be found in the literature (see e.\,g.\ \cite{Alsmeyer}).
\end{remark}

%% file: PFanalytic.tex
We study analytic properties of the operator-valued function $z\mapsto(I-\mathcal{L}_{\eta+z\codefun})^{-1}$, where $\codefun,\eta\in\mathcal F_{\alpha}(\Sigma_A,\mathbb R)$ are fixed and $z\in\mathbb C$.
Below, we collect results from \cite{Pollicott_Ergodictheo,Lalley,ParryPollicott}.

Let $\mathcal B(\mathcal F_{\alpha}(\Sigma_A))$ denote the set of all bounded linear operators on $\mathcal F_{\alpha}(\Sigma_A)$ to $\mathcal F_{\alpha}(\Sigma_A)$. 
Since $\mathcal F_{\alpha}(\Sigma_A)$ endowed with $\|\cdot\|_{\alpha}\defeq\lvert\cdot\rvert_{\alpha}+\|\cdot\|_{\infty}$ is a Banach space, also $(\mathcal B(\mathcal F_{\alpha}(\Sigma_A)),\|\cdot\|_{\text{op}})$ is a Banach space, \cite[p.150]{Kato}. Here, \mbox{$\|\cdot\|_{\text{op}}$} denotes the operator norm, given by $\|B\|_{\text{op}}\defeq\sup_{g\in\mathcal F_{\alpha}(\Sigma_A),\;\|g\|_{\alpha}=1}\|Bg\|_{\alpha}$
for $B\in\mathcal B(\mathcal F_{\alpha}(\Sigma_A))$. 
A function $f\colon D\to\mathcal B(\mathcal F_{\alpha}(\Sigma_A))$ defined on an open domain $D\subset\mathbb C$ is called \emph{holomorphic} if, for all $z\in D$, there exists $l(z)\in\mathcal{B}(\mathcal F_{\alpha}(\Sigma_A))$ such that $\lim_{h\to 0}\|h^{-1}(f(z+h)-f(z))-l(z)\|_{\text{op}}=0$.
Following convention, we interchangeably use the terms holomorphic and \emph{analytic}. 
(For more insights, see \cite{Kato}, especially Ch.~III.3 and VII{.}1.)

For clarity we write $\rho\defeq\eta+z\codefun$ and sometimes $\rho(z)\defeq\eta+z\codefun$ if we want to stress dependence on $z$.
The Ruelle-Perron-Frobenius operator $\mathcal L_{\roz}$ is a bounded linear operator on the Banach space $(\mathcal F_{\alpha}(\Sigma_A),\|\cdot\|_{\alpha})$. 
We write $\textup{spec}(\PF_{\roz})\defeq\{\lambda\in\mathbb C\mid \mathcal L_{\roz}-\lambda I\ \text{is not invertible}\}$ for its \emph{spectrum} and $\text{spr}(\PF_{\roz})$ for its \emph{spectral radius}, i.\,e.\ the radius of the smallest closed disc centred at the origin which contains $\textup{spec}(\PF_{\roz})$.
By the \emph{spectral radius formula},   
\begin{equation}\label{thm:SpectralRadiusFormula}
	\lim_{n\to\infty}\|\PF_{\roz}^n\|_{\text{op}}^{1/n}=\text{spr}(\PF_{\roz}).
\end{equation}
For the following statement let $\Re(z)$ and $\Im(z)$ respectively denote the real and imaginary parts of $z\in\mathbb C$.
\begin{theorem}[\cite{Pollicott_Ergodictheo}]\label{thm:TheoremBLalley}
Let $\alpha\in(0,1)$ and $\codefun,\eta\in\mathcal F_{\alpha}(\Sigma_A)$. Suppose that $z\in\mathbb C\setminus \mathbb R$.
\begin{enumerate}
\item If for some $b\in\mathbb R$ the function $(\Im(z)\codefun-b)/(2\pi)$ is co-homologous to an integer-valued function, then $\ee^{\im b}\eigenv_{_{\eta+\Re(z)\codefun}}$ is a simple eigenvalue of $\PF_{\eta+z\codefun}$, and the rest of the spectrum is contained in a disc centred at zero of radius strictly less than $\eigenv_{_{\eta+\Re(z)\codefun}}$.
\item\label{it:lalleyBother} Otherwise, the entire spectrum of $\PF_{\eta+z\codefun}$ is contained in a disc centred at zero of radius strictly less than $\eigenv_{_{\eta+\Re(z)\codefun}}$.
\end{enumerate}
\end{theorem}

Below, we present useful results that follow from \cite{Lalley} and Thm.~\ref{thm:TheoremBLalley}. 
Note that in \cite{Lalley} only the special case that $\eta$ is the constant zero-function is covered.
However, the proofs work in the same way when $\eta\in\mathcal{F}_{\alpha}(\Sigma_A,\mathbb R)$ is arbitrary. Therefore, we omit the proofs and refer the reader to the respective proofs in \cite{Lalley}.

Results in regular perturbation theory \cite[Sec.~7{.}1 and 4{.}3]{Kato} imply that $z\mapsto\eigenv_{\rho(z)}$, $z\mapsto\eigenf_{\rho(z)}$ and $z\mapsto\nu_{\rho(z)}$ extend to holomorphic functions in a neighbourhood of $\mathbb R$ such that $\eigenv_{\rho(z)}\neq 0$, $\PF_{\rho(z)}\eigenf_{\rho(z)}=\eigenv_{\rho(z)}\eigenf_{\rho(z)}$, $\PF^*_{\rho(z)}\nu_{\rho(z)}=\eigenv_{\rho(z)}\nu_{\rho(z)}$ and $\nu_{\rho(z)}(\eigenf_{\rho(z)})=\nu_0(h_{\rho(z)})=1$ \cite[p.\,27]{Lalley}.

\begin{proposition}[{\cite[Props.~7{.}1 and 7{.}2]{Lalley}}]\label{prop:Lalley7.12}
	Fix $\alpha\in(0,1)$. Let $\codefun,\eta\in\mathcal F_{\alpha}(\Sigma_A,\mathbb R)$ and let $-\mdim$ be the unique real zero of $t\mapsto P(\eta+t\codefun)$. Then
\begin{enumerate}
	\item $z\mapsto(I-\PF_{\roz})^{-1}$ is holomorphic in the half-plane $\Re(z)<-\mdim$.
	\item\label{it:Lalley12_pole} $z\mapsto(I-\PF_{\roz})^{-1}$ has a simple pole at $z=-\mdim$ and for $\codefunc\in\mathcal F_{\alpha}(\Sigma_A)$,
	\begin{equation}\label{eq:LalLzweistrich}
		(I-\PF_{\roz})^{-1}\codefunc=\eigenv_{\roz}(1-\eigenv_{\roz})^{-1}\textstyle{\int} \codefunc\textup{d}\nu_{\roz}\cdot\eigenf_{\roz} + (I-\PF''_{\roz})^{-1}\codefunc,
	\end{equation}
	for $z$ in some punctured neighbourhood of $z=-\mdim$, where 
	\begin{align*}
		\PF''_{\roz}&\defeq\PF_{\roz}-\PF'_{\roz}\qquad\qquad\text{with}\\
		\PF'_{\roz}\codefunc&\defeq\eigenv_{\roz}\textstyle{\int} \codefunc\textup{d}\nu_{\roz}\cdot\eigenf_{\roz}.
	\end{align*}
	Moreover, $z\mapsto (I-\PF''_{\roz})^{-1}$ is a holomorphic operator-valued function in a neighbourhood of $z=-\mdim$.
    \end{enumerate}
\end{proposition}
The factor $\eigenv_{\roz}$ of the first summand of \eqref{eq:LalLzweistrich} is missing in \cite{Lalley}. However, the relevant $z$-value is $-\mdim$, and $\eigenv_{-\mdim}=1$.

We are interested in the residue of $z\mapsto(I-\PF_{\rho(z)})^{-1}$ at the simple pole $z=-\mdim$. For this, we use that the topological pressure function $t\mapsto P(\eta+t\codefun)$ is real-analytic for $t\in\mathbb R$ and real-valued $\codefun,\eta\in\mathcal F_{\alpha}(\Sigma_A,\mathbb R)$ and that it satisfies 
\begin{equation}\label{eq:analyticpressure}
	\frac{\text{d}}{\text{d}t}P(\eta+t\codefun)=\int \codefun\textup{d}\mu_{\eta+t\codefun},\quad t\in\mathbb R.
\end{equation}
The analyticity of $z\mapsto(I-\PF_{\rho(z)})^{-1}$ can be proved with methods of analytic perturbation theory as presented in \cite{Kato}. This method of proof is due to \cite{Ruelle_thermodynamic}.
Further, since $z\mapsto\eigenv_{\rho(z)}$ has an analytic continuation to a neighbourhood of $\mathbb R$ and $P(\codefun)=\log(\eigenv_{\codefun})$ for real-valued $\codefun\in\mathcal{C}(\Sigma_A,\mathbb R)$, we can extend $P$ analytically by setting $P(\rho(z))\defeq\log(\eigenv_{\rho(z)})$. `Formally this definition can only be made modulo $2\pi \im$ since $\log$ is multiple valued, although we shall ask that $P(\codefun)$ be real-valued when $\codefun$ is real-valued' \cite[p.\,31]{ParryPollicott}.
In this way, \eqref{eq:analyticpressure} extends to a neighbourhood of $\mathbb R$.
Combined with \eqref{eq:analyticpressure}, Prop.~\ref{prop:Lalley7.12}\ref{it:Lalley12_pole} yields the following corollary since $z\mapsto \eigenv_{\rho(z)}$, $z\mapsto \int \codefunc\textup{d}\nu_{\rho(z)}$ and $z\mapsto\eigenf_{\rho(z)}$ are continuous at $z=-\mdim$.

\begin{corollary}\label{cor:Lalley7.2}
	Let $\codefun\in\mathcal{C}(\Sigma_A,\mathbb R)$ and $\codefunc\in\mathcal F_{\alpha}(\Sigma_A,\mathbb R)$. Then, for $x\in\Sigma_A$, the residue of $(I-\PF_{\eta+z\codefun})^{-1}\codefunc(x)$ at $z=-\mdim$ is
\[
	-\frac{\int \codefunc\textup{d}\nu_{\eta-\mdim \codefun}}{\int \codefun\textup{d}\mu_{\eta-\mdim \codefun}}\eigenf_{\eta-\mdim \codefun}(x).
\]
\end{corollary}
The residue is given in \cite[p.27]{Lalley}; however, with a different sign.

\begin{proposition}[{\cite[Prop.~7{.}3]{Lalley}}]\label{prop:Lalley7.3}
	If $\codefun$ is non-lattice then $z\mapsto(I-\PF_{\eta+z\codefun})^{-1}$ is holomorphic in a neighbourhood of every $z$ on the line $\Re(z)=-\mdim$ except for $z=-\mdim$.
\end{proposition}

\begin{proposition}[{\cite[Prop.~7{.}4]{Lalley}}]\label{prop:Lalley7.4}
	If $\codefun$ is integer-valued but not co-homologous to any function valued in a proper subgroup of the integers, then $z\mapsto(I-\PF_{\eta+z\codefun})^{-1}$ is $2\pi\im$-periodic, and holomorphic at every $z$ on the line $\Re(z)=-\mdim$ such that $\Im(z)/(2\pi)$ is not an integer.
\end{proposition}

%% file: renewalproof.tex
\subsubsection{Proof of the Renewal Thm.~\ref{thm:RT1}}
As the statement of Thm.~\ref{thm:RT1} suggests, we need to distinguish between the cases of $\codefun$ being lattice or non-lattice. We start with the non-lattice situation, for which we use a useful smoothing argument for showing the desired asymptotic, see \cite[ pp.\,29 ff.]{Lalley}. 
For a probability density $\prob\colon\mathbb R\to\mathbb R$ we consider its Fourier-Laplace transform given by 
\[
	\hat{\prob}(\im\theta)\defeq\int_{-\infty}^{\infty}\ee^{\im\theta t}\prob(t)\textup{d}t
\]
and introduce the following class of probability densities. 
\begin{align*}
	&\mathfrak{P}\defeq \{\prob\colon\mathbb R\to\mathbb R\mid \prob\ \text{is a probability density},   \prob(t)=\prob(-t)\ \text{for}\ t\in\mathbb R.\\
	&\hphantom{\mathfrak{P}\defeq \{\prob\colon\mathbb R\to\mathbb R\mid\ }\hat{\prob}(\im\theta)\ \text{is non-negative, $\mathcal{C}^{\infty}$ and has compact support}\}
\end{align*}
Note that the function $\hat{\prob}\colon\mathbb R\to\mathbb C$ given by 
\begin{equation}\label{eq:prob1}
	\hat{\prob}(\im\theta)\defeq
	\begin{cases}
		\exp\left(\frac{-\theta^2}{1-\theta^2}\right) & \colon\lvert\theta\rvert\leq 1,\\
		0 & \colon\text{otherwise}
	\end{cases}
\end{equation}
defines an even probability density $\prob\colon\mathbb R\to\mathbb R$ which lies in $\mathfrak{P}$. That $\prob$ is a probability density is due to Bochner's theorem, see e.\,g.\ \cite[Satz 15{.}29]{Klenke}. 
Thus, $\mathfrak{P}\neq\emptyset$. 
For the following, fix $\prob$\index{zzpi@$\prob$} as such. As $\prob$ is an even probability density we know that for all $\e>0$ there exists $\tau>0$ such that
\[
	\int_{-\tau}^{\tau}\prob(t)\textup{d}t\geq 1-\e.
\]
For each $\e>0$ fix such a $\tau=\tau(\e)$.
Thus, $\prob_{\e}$ which for $\e>0$ is defined by 
\begin{equation}
	\prob_{\e}(t)\defeq\tfrac{\tau(\e)}{\e}\prob\left(t\tfrac{\tau(\e)}{\e}\right)
\end{equation}
satisfies
\[
	\int_{-\e}^{\e} \prob_{\e}(t)\textup{d}t
	=\int_{-\tau(\e)}^{\tau(\e)}\prob(t)\textup{d}t
	\geq 1-\e.
\]
Moreover, it can be verified that $\prob_{\e}\in\mathfrak{P}$ for all $\e>0$. The smoothing argument is as follows.

\begin{lemma}\label{lem:suffices}
  If for each sufficiently small $\e>0$
  \begin{equation}\label{eq:smoothing}
        \lim_{r\to\infty}\int_{-\infty}^{\infty}\prob_{\e}(r-T)\ee^{-T\mdim}N(T,x)\textup{d}T=G(x)
  \end{equation}
  uniformly for $x\in\Sigma_A$, then the statement of Thm.~\ref{thm:RT1}\ref{it:RT1:nl} holds.
\end{lemma}
\begin{proof}
	In the proof we distinguish between the different cases of \ref{it:regular}. 
	
	\emph{Case \ref{it:monotonic}}: Here, $N(t,x)$ is a monotonic function in $t$ and the statement follows from the proof of \cite[Lemma 8.2]{Lalley}. 
	
	\emph{Case \ref{it:nonoscill}}:
   For $r\in\mathbb R$ and $\e>0$, Condition \ref{it:boundedC} implies that
   \begin{align}\label{eq:A1}
        \left\lvert
	\int_{-\infty}^{\infty}\prob_{\e}(r-T)\ee^{-T\mdim}N(T,x)\textup{d}T - \int_{r-\e}^{r+\e}\prob_{\e}(r-T)\ee^{-T\mdim}N(T,x)\textup{d}T
	\right\rvert
\leq\mathfrak C\e,
   \end{align}
   which tends to 0 as $\e\to 0$ uniformly for $x\in\Sigma_A$. 
   Moreover, we observe that 
   \begin{align}
        \inf_{\tilde{\e}\in[0,2\e)}\ee^{-(r-\tilde{\e})\mdim}N(r-\tilde{\e},x)(1-\e)
	&\leq \int_{r-2\e}^{r}\prob_{\e}(r-\e-T)\ee^{-T\mdim}N(T,x)\textup{d}T\nonumber\\
	&\leq\sup_{\tilde{\e}\in[0,2\e)}\ee^{-(r-\tilde{\e})\mdim}N(r-\tilde{\e},x).\label{eq:infsup}
   \end{align}
   These observations imply that
   \begin{align*}
     G(x)
     \stackrel{\eqref{eq:smoothing}}{=}& \lim_{r\to\infty}\int_{-\infty}^{\infty}\prob_{\e}(r-T)\ee^{-T\mdim}N(T,x)\textup{d}T\\
     =\ \,&\liminf_{\e\searrow 0}\limsup_{r\to\infty}\int_{-\infty}^{\infty}\prob_{\e}(r-T)\ee^{-T\mdim}N(T,x)\textup{d}T\\
     \stackrel{\eqref{eq:A1}}{\leq}&\liminf_{\e\searrow 0}\limsup_{r\to\infty}\int_{r-2\e}^{r}\prob_{\e}(r-\e-T)\ee^{-T\mdim}N(T,x)\textup{d}T\\
     \stackrel{\eqref{eq:infsup}}{\leq}&\liminf_{\e\searrow 0}\limsup_{r\to\infty}\sup_{\tilde{\e}\in[0,2\e)}\ee^{-(r-\tilde{\e})\mdim}N(r-\tilde{\e},x)
   \end{align*}
   and likewise that
   \begin{align*}
     G(x)
     &\geq \limsup_{\e\searrow 0}\liminf_{r\to\infty}\inf_{\tilde{\e}\in[0,2\e)}\ee^{-(r-\tilde{\e})\mdim}N(r-\tilde{\e},x).
   \end{align*}
   Using \ref{it:nonoscill} and the inequalities 
   \begin{align*}
        \inf_{\tilde{\e}\in[0,2\e)}\ee^{-(r-\tilde{\e})\mdim}N(r-\tilde{\e},x)
        \leq \ee^{-r\mdim}N(r,x)
        \leq \sup_{\tilde{\e}\in[0,2\e)}\ee^{-(r-\tilde{\e})\mdim}N(r-\tilde{\e},x)
   \end{align*}
   we conclude that uniformly for $x\in\Sigma_A$
   \begin{align*}
        G(x)
	&= \lim_{r\to\infty}\ee^{-r\mdim}N(r,x).
    \end{align*}
    
    \emph{Case \ref{it:nldRi}}:  For $x\in\Sigma_A$ and $t\in\mathbb R$ set 
\begin{equation*}
	\renfcnn_x(t)\defeq\ee^{-t\mdim}\renfcn_x(t).
\end{equation*}
Then $\{\lvert \renfcnn_x\rvert\mid x\in\Sigma_A\}$ is equi d.\,R.\,i.\ by \ref{it:nldRi}. Moreover,
\begin{align}
\begin{aligned}\label{eq:renfcnn}
	\ee^{-t\mdim}N(t,x)
	&\defeq \sum_{n=0}^{\infty}\sum_{y:\sigma^n y=x}\codefunc(y)\renfcn_y(t-S_n\codefun(y))\ee^{S_n\eta(y)}\ee^{-t\mdim} \\
	&= \sum_{n=0}^{\infty}\sum_{y:\sigma^n y=x}\codefunc(y)\renfcnn_y(t-S_n\codefun(y))\ee^{S_n(\eta-\mdim\codefun)(y)}.
\end{aligned}
\end{align}
For showing that \eqref{eq:smoothing} implies $\lim_{r\to\infty}\ee^{-r\mdim}N(r,x)=G(x)$, we consider
  \[\begin{array}[h]{rcl}
    A_{\e}(r,x)&
    \defeq&\left\lvert \displaystyle{\int_{-\infty}^{\infty}}\prob_{\e}(r-T)\ee^{-T\mdim}N(T,x)\textup{d}T-\ee^{-r\mdim}N(r,x)\right\rvert\\
    &\stackrel{\eqref{eq:A1}}{\leq}& \left\lvert \displaystyle{\int_{r-2\e}^{r}}\prob_{\e}(r-\e-T)\ee^{-T\mdim}N(T,x)\textup{d}T-\displaystyle{\int_{r-2\e}^r}\prob_{\e}(r-\e-T)\ee^{-r\mdim}N(r,x)\textup{d}T\right\rvert\\
    && \hspace{0.8cm} +\mathfrak{C}\e + \left\lvert \ee^{-r\mdim}N(r,x) \right\rvert\e\\
    &\stackrel{\ref{it:boundedC},\eqref{eq:renfcnn}}{\leq}& \bigg{\lvert} \displaystyle{\int_{r-2\e}^r}\prob_{\e}(r-\e-T)\sum_{n=0}^{\infty}\sum_{y:\sigma^ny=x}\codefunc(y)\ee^{S_n(\eta-\mdim\codefun)(y)}\\
	&&\hspace{0.8cm}\times\big{[}\renfcnn_y(T-S_n\codefun(y))-\renfcnn_y(r-S_n\codefun(y))\big{]}\textup{d}T\bigg{\rvert}
   +2\mathfrak{C}\e \\
    &\leq&  \displaystyle{\sum_{n=0}^{\infty}\sum_{y:\sigma^ny=x}}\codefunc(y)\ee^{S_n(\eta-\mdim\codefun)(y)}\\
	&&\hspace{0.8cm}\times\displaystyle{\int_{r-2\e}^r}\prob_{\e}(r-\e-T)\Big{\lvert} \renfcnn_y(T-S_n\codefun(y))-\renfcnn_y(r-S_n\codefun(y))\Big{\rvert}\textup{d}T 
    +2\mathfrak{C}\e.
  \end{array}\]
For the last inequality we have used the monotone convergence theorem. 

Set 
      \[
      d^{2\e}(t)\defeq\sup_{y\in\Sigma_A}\left(\sup_{\tilde{\e}\in[0,2\e)}\renfcnn_y(t-\tilde{\e})-\inf_{\tilde{\e}\in[0,2\e)}\renfcnn_y(t-\tilde{\e})\right).
      \]
Since $\{\renfcnn_x\mid x\in\Sigma_A\}$ is equi d.\,R.\,i.\ we know, for sufficiently small $\e>0$, that
\begin{equation}\label{eq:d2e}
  \sum_{k\in\mathbb Z}d^{2\e}(k\cdot 2\e)\quad\text{exists and}
  \quad \lim_{\e\to 0}\sum_{k\in\mathbb Z}d^{2\e}(k\cdot 2\e)\cdot 2\e=0.
\end{equation}
Therefore, we may deduce the following chain of inequalities.
\begin{align*}
  A_{\e}(r,x)-2\mathfrak{C}\e
  &\leq \sum_{n=0}^{\infty}\sum_{y:\sigma^ny=x}\codefunc(y)\ee^{S_n(\eta-\mdim\codefun)(y)} d^{2\e}(r-S_n\codefun(y))\\
  &\leq \sum_{k\in\mathbb Z} d^{2\e}(k\cdot 2\e)\|\codefunc\|_{\infty}\sum_{n=0}^{\infty}\sum_{y:\sigma^ny=x}\ee^{S_n(\eta-\mdim\codefun)(y)} \mathds 1_{((k-1)2\e,(k+1)2\e]}(r-S_n\codefun(y)) \\
  &\leq c\|\codefunc\|_{\infty} \sum_{k\in\mathbb Z} d^{2\e}(k\cdot 2\e)\underbrace{\sum_{n=0}^{\infty}\sum_{y:\sigma^ny=x} \mu_{\eta-\mdim\codefun}([y\vert_n]) \mathds 1_{((k-1)2\e,(k+1)2\e]}(r-S_n\codefun(y))}_{\eqdef a_{\e,k}(r,x)},
\end{align*}
where for the last estimate we have used the Gibbs property \eqref{eq:Gibbs} of $\mu_{\eta-\mdim\codefun}$ with constant $c$.
For simplicity of presentation we first treat the case that $\codefun$ is strictly positive. In this case there exists $\kappa>0$ such that $\codefun(x)\geq\kappa$ for all $x\in\Sigma_A$ since $\codefun$ is continuous and $\Sigma_A$ is compact. As we will later be interested in the limiting behaviour when $\e\to 0$, we can freely assume that $4\e\leq\kappa$ and set $m\defeq\lfloor \kappa/(2\e)\rfloor$. Note that $m\geq 2$. Let $\sigma^{-n}(x)$ denote the set of pre-images of $x$ under $\sigma^n$, that is
\begin{equation*}
  \sigma^{-n}(x)\defeq\{y\in\Sigma_A\mid \sigma^n y=x\}.
\end{equation*}
Let $y\in\sigma^{-n}(x)$ and assume that $r-S_n\codefun(y)\in((k-1)2\e,(k+1)2\e]\eqdef I_k(\e)$. Then $\sigma y\in\sigma^{-(n-1)}(x)$ and
\begin{align*}
  r-S_{n-1}\codefun(\sigma y)
  &=r-S_n\codefun(y)+\codefun(y)
  \geq r-S_n\codefun(y)+\kappa
  > (k-1)2\e + 2m\e\\
  &=((k-2+m)+1)2\e,
\end{align*}
whence $r-S_{n-1}\codefun(\sigma y)\notin I_{k+q}(\e)$ for $q\in\{0,\ldots,m-2\}$.
Now let $\tilde{y}\in\sigma^{-(n+1)}(x)$ be such that $\sigma\tilde{y}=y$. Then
\begin{align*}
  r-S_{n+1}\codefun(\tilde{y})
  &=r-S_n\codefun(y)-\codefun(\tilde{y})
  \leq r-S_n\codefun(y)-\kappa
  \leq (k+1)2\e - 2m\e\\
  &=((k-m+2)-1)2\e,
\end{align*}
whence $r-S_{n+1}\codefun(\tilde{y})\notin I_{k-q}(\e)$ for $q\in\{0,\ldots,m-2\}$.
 Therefore,
\begin{align}
  &\sum_{q=0}^{m-2} a_{\e,k+q}(r,x)\label{eq:lessthanone}\\
  &\qquad= \sum_{q=0}^{m-2} \sum_{n=0}^{\infty}\sum_{y:\sigma^ny=x} \mu_{\eta-\mdim\codefun}([y\vert_n]) \mathds 1_{((k+q-1)2\e,(k+q+1)2\e]}(r-S_n\codefun(y))
  \leq 1,\nonumber
\end{align}
as $\mu_{\eta-\mdim\codefun}$ is a probability measure on $\Sigma_A$.

Next, we show existence of $q\in\{0,\ldots, m-2\}$ such that, for all $\e<\kappa/8$,
\begin{align}\label{eq:dineq}
   \sum_{k\in\mathbb Z} d^{2\e}(k\cdot 2\e)a_{\e,k+q}(r,x)
   \leq \frac{2}{\kappa}\cdot 2\e\sum_{k\in\mathbb Z} d^{2\e}(k\cdot 2\e).
\end{align}
Assume that this was not the case. Since for $\e<\kappa/8$ we have that $1<2(\kappa-4\e)/\kappa<2\e(m-1)\cdot2/\kappa$, it follows that
\begin{align*}
   \sum_{k\in\mathbb Z} d^{2\e}(k\cdot 2\e)
   &< 2\e(m-1)\frac{2}{\kappa} \sum_{k\in\mathbb Z} d^{2\e}(k\cdot 2\e)
   = \sum_{q=0}^{m-2}2\e\frac{2}{\kappa} \sum_{k\in\mathbb Z} d^{2\e}(k\cdot 2\e)\\
   &\hspace{-0.3cm}\stackrel{\neg\eqref{eq:dineq}}{<}\sum_{q=0}^{m-2} \sum_{k\in\mathbb Z} d^{2\e}(k\cdot 2\e)a_{\e,k+q}(r,x)
   \stackrel{\eqref{eq:lessthanone}}{\leq}\sum_{k\in\mathbb Z} d^{2\e}(k\cdot 2\e).
\end{align*}
As $\sum_{k\in\mathbb Z} d^{2\e}(k\cdot 2\e)$ exists for sufficiently small $\e$, it follows that \eqref{eq:dineq} must be true for some $q\in\{0,\ldots,m-2\}$. 

Notice $a_{\e,k}(r-2q\e,x)=a_{\e,k+q}(r,x)$, and so for this $q$,
\begin{align}
\begin{aligned}\label{eq:Arx}
   \limsup_{r\to\infty} A_{\e}(r,x)
   &=\limsup_{r\to\infty} A_{\e}(r-2q\e,x)\\
   &\leq c\|\codefunc\|_{\infty}\limsup_{r\to\infty}\sum_{k\in\mathbb Z} d^{2\e}(k\cdot 2\e)\underbrace{a_{\e,k}(r-2q\e,x)}_{=a_{\e,k+q}(r,x)}+2\mathfrak{C}\e\\
   &\hspace{-0.2cm}\stackrel{\eqref{eq:dineq}}{\leq}c\|\codefunc\|_{\infty} \frac{2}{\kappa}2\e\sum_{k\in\mathbb Z} d^{2\e}(k\cdot 2\e)+2\mathfrak{C}\e.
\end{aligned}
\end{align}

All in all we have for each sufficiently small $\e>0$
\begin{align*}
  &\limsup_{r\to\infty}\left\lvert G(x) - \ee^{-r\mdim}N(r,x)\right\rvert\\
	&\qquad\ \leq \limsup_{r\to\infty}\left\lvert G(x)-\int_{-\infty}^{\infty}\prob_{\e}(r-T)\ee^{-T\mdim}N(T,x)\textup{d}T \right\rvert\\
	&\qquad\qquad + \limsup_{r\to\infty}\left\lvert  \int_{-\infty}^{\infty}\prob_{\e}(r-T)\ee^{-T\mdim}N(T,x)\textup{d}T - \ee^{-r\mdim}N(r,x)\right\rvert\\
   &\qquad\stackrel{\eqref{eq:smoothing}}{=} \limsup_{r\to\infty} A_{\e}(r,x)\\
   &\qquad\stackrel{\eqref{eq:Arx}}{\leq}c \|\codefunc\|_{\infty}\frac{2}{\kappa} 2\e \sum_{k\in\mathbb Z} d^{2\e}(k\cdot 2\e) +2\mathfrak{C}\e.
\end{align*}
This implies that
\begin{align*}
  \limsup_{r\to\infty}\left\lvert G(x) - \ee^{-r\mdim}N(r,x)\right\rvert
  &\leq \limsup_{\e\to 0}\left(c\|\codefunc\|_{\infty}2\e\frac{2}{\kappa}  \sum_{k\in\mathbb Z} d^{2\e}(k\cdot 2\e) +2\mathfrak{C}\e\right)\\
  &\hspace{-0.2cm}\stackrel{\eqref{eq:d2e}}{=}0.
\end{align*}

Now we treat the case that $\codefun$ is not strictly positive.
By assumption there exists $n\in\mathbb N$ for which $S_n\codefun$ is strictly positive. Since $\codefun$ is continuous and $\Sigma_A$ is compact, there exists a $\tilde{\kappa}>0$ for which $S_n\codefun\geq\tilde{\kappa}$. For $l\in\mathbb N$ and $j\in\{0,\ldots,n-1\}$ this implies that 
\begin{align}
	S_{ln+j}\codefun(x)
	&=\sum_{i=0}^{ln+j-1}\codefun(\sigma^i x)
	= \sum_{m=0}^{l-1}\sum_{i=mn}^{(m+1)n-1}\codefun(\sigma^i x)+\sum_{i=ln}^{ln+j-1}\codefun(\sigma^ix)\label{eq:Snpositive}\\
	&=\sum_{m=0}^{l-1} S_n\codefun(\sigma^i(\sigma^{mn}x)) +\sum_{i=ln}^{ln+j-1}\codefun(\sigma^ix)
	\geq l\tilde{\kappa}+\inf_{\substack{x\in\Sigma_A \\ 0\leq i\leq n-1}}S_i\codefun(x)\nonumber
\end{align}
From this we can conclude that there exists $m^*\in\mathbb N$ such that $S_m\codefun$ is strictly positive for all $m\geq m^*$. Thus, there exists $\kappa>0$ with $S_m\codefun\geq\kappa$ for all $m\geq m^*$.
In the same way as in the case that $\codefun$ is strictly positive we can show that if $y\in\sigma^{-n}(x)$ with $r-S_n\codefun(y)\in I_k(\e)$ then $r-S_{n-m^*}\codefun(\sigma^{m^*} y)\notin I_{k+q}(\e)$ for $q\in\{0,\ldots,\lfloor\kappa/(2\e)\rfloor-2\}$ whenever $n\geq m^*$ and $r-S_{n+m^*}\codefun(\tilde{y})\notin I_{k-q}(\e)$ for all $q\in\{0,\ldots,\lfloor\kappa/(2\e)\rfloor-2\}$ where $\tilde{y}\in\sigma^{-(n+m^*)}(x)$ is such that $\sigma^{m^*}\tilde{y}=y$. This implies that 
\begin{align*}
  \sum_{q=0}^{\lfloor\kappa/(2\e)\rfloor-2}\sum_{n=0}^{\infty}\sum_{y:\sigma^ny=x} \mu_{\eta-\mdim\codefun}([y\vert_n]) \mathds 1_{((k+q-1)2\e,(k+q+1)2\e]}(r-S_n\codefun(y))
    \leq 2m^*.
\end{align*}
The remainder of the proof follows in the same way as in the case that $\codefun$ is strictly positive.
\end{proof}

\begin{proof}[Proof of Thm.~\ref{thm:RT1}\ref{it:RT1:nl}]
	Inspired by \cite[Thm.~1]{Lalley} we consider for $x\in\Sigma_A$ the Fourier-Laplace transform of $t\mapsto\ee^{-t\mdim}N(t,x)$ at $z\in\mathbb C$:
	\begin{equation}
		\Laplace(z,x)\defeq\int_{-\infty}^{\infty}\ee^{zT}\ee^{-T\mdim}N(T,x)\textup{d}T.
	\end{equation}
	Conditions \ref{it:boundedC} and \ref{it:boundedneg} imply that $\Laplace(\cdot,x)\colon\mathbb C\to\mathbb C$, $z\mapsto \Laplace(z,x)$ is well-defined and analytic on
	$\{z\in\mathbb C\mid -s<\Re(z)<0\}$.
        What is more, for small enough $\e>0$, Conditions \ref{it:boundedC} and \ref{it:boundedneg} imply that $\Laplace(\cdot,x)$ converges absolutely and uniformly on
	\[
		\{z\in\mathbb C\mid -s+\e\leq\Re(z)\leq-\e\}.
	\]
	Now, in every such region, using \ref{it:boundedC} as well as the monotone and dominated convergence theorems, we obtain
	\begin{align*}
	  \Laplace(z,x)
	  &=\sum_{n=0}^{\infty}\sum_{y:\sigma^n y=x}\ee^{S_n(\eta+(z-\mdim)\codefun)(y)}\codefunc(y)\int_{-\infty}^{\infty}\ee^{(z-\mdim)T}\renfcn_y(T)\textup{d}T\\
          &=\sum_{n=0}^{\infty}\PF_{\eta+(z-\mdim)\codefun}^n\left(\codefunc\int_{-\infty}^{\infty}\ee^{(z-\mdim)T}\renfcn_{\cdot}(T)\textup{d}T\right)(x),
	\end{align*}
	where $\renfcn_{\cdot}(T)\colon\Sigma_A\to\mathbb R$, $x\mapsto \renfcn_x(T)$.
        Note that Conditions \ref{it:Lebesgueintegrable} to \ref{it:boundedneg} imply that $\int_{-\infty}^{\infty}\ee^{(z-\mdim)T}\renfcn_{x}(T)\textup{d}T$ exists if $-s<\Re(z)\leq 0$, since $\codefunc(x)\lvert\renfcn_x(t)\rvert\leq N^{\text{abs}}(t,x)$.
        Thus, by Thm.~\ref{thm:TheoremBLalley}\ref{it:lalleyBother}, the spectral radius formula \eqref{thm:SpectralRadiusFormula} and the fact that $\eigenv_{\eta-\mdim\xi}=1$ (see Prop.~\ref{thm:eigenvalueone}), the above series converges for $-s<\Re(z)<0$, and we obtain
	\[
		\Laplace(z,x)
		=(I-\PF_{\eta+(z-\mdim)\codefun})^{-1}\left(\codefunc\int_{-\infty}^{\infty}\ee^{(z-\mdim)T}\renfcn_{\cdot}(T)\textup{d}T\right)(x).
	\]
	By Prop.~\ref{prop:Lalley7.3}, the operator-valued function $z\mapsto (I-\PF_{\eta+(z-\mdim)\codefun})^{-1}$ is holomorphic at every $z$ on the line $\Re(z)=0$ except for $z=0$, which is a simple pole by Prop.~\ref{prop:Lalley7.12}. Thus, according to Cor.~\ref{cor:Lalley7.2} the residue of $z\mapsto\Laplace(z,x)$ at $z=0$ is
	\begin{equation}
                -\frac{\int_{\Sigma_A}\codefunc(y)\int_{-\infty}^{\infty}\ee^{-T\mdim}\renfcn_y(T)\textup{d}T\textup{d}\nu_{\eta-\mdim\codefun}(y)}{\int\codefun\textup{d}\mu_{\eta-\mdim\codefun}}\eigenf_{\eta-\mdim\codefun}(x)=-G(x),
	\end{equation}
        where $G$ is as in Thm.~\ref{thm:RT1}\ref{it:RT1:nl}.
	Thus, $\Laplace(z,x)$ has the following representation.
	\begin{equation}\label{eq:Laplaceq}
		\Laplace(z,x)=q(z,x)-\frac{G(x)}{z},
	\end{equation}
	where $q(\cdot,x)\colon\mathbb C\to\mathbb C$, $z\mapsto q(z,x)$ is holomorphic in a region containing the strip $\{z\in\mathbb C\mid -s+\e\leq\Re(z)\leq 0\}$ with sufficiently small $\e>0$. Conditions \ref{it:boundedC}, \ref{it:boundedneg} and Lebesgue's dominated convergence theorem now imply for every $\e\in(0,1]$ that
	\begin{align}
	\begin{aligned}\label{eq:domconv}
		&\int_{-\infty}^{\infty}\prob_{\e}(r-T)\ee^{-T\mdim}N(T,x)\textup{d}T\\
		&\quad=\lim_{\beta\searrow 0}\int_{-\infty}^{\infty}\prob_{\e}(r-T)\ee^{-T(\mdim+\beta)}N(T,x)\textup{d}T.
	\end{aligned}
	\end{align}
	Using the inverse Fourier-Laplace transform $\prob_{\e}(t)=\int_{-\infty}^{\infty}\ee^{-\im\theta t}\hat{\prob_{\e}}(\im\theta)\textup{d}\theta/(2\pi)$ and that the integral on the left hand side of \eqref{eq:domconv} exists, we can convert the integral from the right hand side of \eqref{eq:domconv} for sufficiently small $\beta>0$ as follows.
	\begin{align}
	\begin{aligned}\label{eq:parseval}
		&\int_{-\infty}^{\infty}\prob_{\e}(r-T)\ee^{-\beta T-T\mdim}N(T,x)\textup{d}T\\
		&\quad= \int_{-\infty}^{\infty}\int_{-\infty}^{\infty}\ee^{-\im\theta(r-T)}\hat{\prob_{\e}}(\im\theta)\frac{\textup{d}\theta}{2\pi}\ee^{-\beta T-T\mdim}N(T,x)\textup{d}T\\
		&\quad= \int_{-\infty}^{\infty}\int_{-\infty}^{\infty}\ee^{(\im\theta-\beta)T}\ee^{-T\mdim}N(T,x)\hat{\prob_{\e}}(\im\theta)\ee^{-\im\theta r}\frac{\textup{d}\theta}{2\pi}\textup{d}T\\
		&\quad=\int_{-\infty}^{\infty}\Laplace(\im\theta-\beta,x)\hat{\prob_{\e}}(\im\theta)\ee^{-\im\theta r}\frac{\textup{d}\theta}{2\pi}\\
		&\ \stackrel{(\ref{eq:Laplaceq})}{=} \int_{-\infty}^{\infty}\left(q(\im\theta-\beta,x)+\frac{G(x)(\im\theta+\beta)}{\theta^2+\beta^2}\right)
		\hat{\prob_{\e}}(\im\theta)\ee^{-\im\theta r}\frac{\textup{d}\theta}{2\pi}.
	\end{aligned}
	\end{align}
	The measures given by $\frac{\beta}{\pi(\theta^2+\beta^2)}\textup{d}\theta$ converge weakly to the Dirac point-mass at zero as $\beta\to 0$ \cite[p.\,31]{Lalley}. Moreover, the imaginary part on the right hand side of \eqref{eq:parseval} can be ignored since the left hand side is real. Using that $\hat{\prob_{\e}}(\im\theta)$ is real and that $\hat{\prob_{\e}}(0)=1$ for all $\e\in(0,1]$, we obtain
	\begin{align}
		&\lim_{\beta\searrow 0}\int_{-\infty}^{\infty}\prob_{\e}(r-T)\ee^{-\beta T-T\mdim}N(T,x)\textup{d}T\label{eq:erg}\\
		&=\Re\left(\int_{-\infty}^{\infty}q(\im\theta,x)\hat{\prob_{\e}}(\im\theta)\ee^{-\im\theta r}\frac{\textup{d}\theta}{2\pi}+\frac{G(x)}{2} + G(x)\int_{-\infty}^{\infty}\hat{\prob_{\e}}(\im\theta)\frac{\im\ee^{-\im\theta r}}{\theta}\frac{\textup{d}\theta}{2\pi}\right).\nonumber
	\end{align}
        We separately treat the two integrals on the right hand side of \eqref{eq:erg} and begin with the first one.
        Recall that $\hat{\prob_{\e}}(\im\theta)=\hat{\prob}(\im\theta\e/\tau(\e))$ is $\mathcal C^{\infty}$ and has compact support, which is contained in $[-\tau(\e)/\e,\tau(\e)/\e]\eqdef[-S,S]$. Also, recall that $q(\cdot,x)$ is analytic in a neighbourhood of $[-\im S,\im S]$ and continuous in $x$. As mentioned in \cite[p.\,31f.]{Lalley}, the Cauchy integral formula for derivatives implies that $\frac{\textup{d}}{\textup{d}z}q(z,x)\vert_{z=\im\theta}$ is uniformly continuous in $\theta$ and hence bounded on $[- S, S]\times\Sigma_A$. Thus, $\frac{\textup{d}}{\textup{d}z}q(z,x)$ is bounded on $[-\im S,\im S]\times\Sigma_A$. Integration by parts now implies that
	\begin{align}
	\begin{aligned}\label{eq:partialintegration}
		&\int_{-S}^{S} q(\im\theta,x)\hat{\prob_{\e}}(\im\theta)\ee^{-\im\theta r}\frac{\textup{d}\theta}{2\pi}\\
		&\quad= \im q(\im\theta,x)\hat{\prob_{\e}}(\im\theta)\frac{\ee^{-\im \theta r}}{2\pi r}\Big{\vert}_{\theta=-S}^{S}
			+\im\int_{-S}^{S}\frac{\textup{d}}{\textup{d}\theta}\left(q(\im\theta,x)\hat{\prob_{\e}}(\im\theta)\right)\frac{\ee^{-\im \theta r}}{2\pi r}\textup{d}\theta.
	\end{aligned}
	\end{align}
	As the support of $\hat{\prob_{\e}}$ is contained in $[-\im S,\im S]$ and $\hat{\prob_{\e}}$ is $\mathcal C^{\infty}$, the first term on the right hand side of \eqref{eq:partialintegration} equals zero for all $r>0$. For the second term on the right hand side of \eqref{eq:partialintegration} we use that the definition of $\hat{\prob}$ given in \eqref{eq:prob1} and the fact that $\hat{\prob_{\e}}(\im\theta)=\hat{\prob}(\im\theta\e/\tau(\e))$ imply that $\hat{\prob_{\e}}(\im\theta)$ and $\frac{\textup{d}}{\textup{d}\theta}\hat{\prob_{\e}}(\im \theta)$ are uniformly bounded for $\e\in(0,1]$. This shows that the second term on the right hand side of \eqref{eq:partialintegration} converges to zero uniformly for $\e\in(0,1]$ and $x\in\Sigma_A$ as $r\to\infty$.
        Thus, 
        \begin{equation}\label{eq:Req}
        \lim_{r\to\infty}\Re\left(\int_{-\infty}^{\infty}q(\im\theta,x)\hat{\prob_{\e}}(\im\theta)\ee^{-\im\theta r}\frac{\textup{d}\theta}{2\pi}\right)=0
        \end{equation}
        uniformly for $x\in\Sigma_A$.
        Now, we consider the second integral on the right hand side of \eqref{eq:erg}:
        \begin{align*}
        \Re\left(\int_{-\infty}^{\infty}\hat{\prob_{\e}}(\im\theta)\frac{\im\ee^{-\im\theta r}}{\theta}\frac{\textup{d}\theta}{2\pi}\right)
        &= \int_{- S}^{S}\hat{\prob_{\e}}(\im\theta)\frac{\sin(\theta r)}{\theta}\frac{\textup{d}\theta}{2\pi}
        = \int_{0}^{S}\hat{\prob_{\e}}(\im\theta)\frac{\sin(\theta r)}{\theta}\frac{\textup{d}\theta}{\pi}.
        \end{align*}
        Using the sine integral $\texttt{Si}(t)\defeq\int_0^{t}\frac{\sin(\theta)}{\theta}\textup{d}\theta$ and $\lim_{t\to\infty}\texttt{Si}(t)=\pi/2$ we infer that 
        \begin{align*}
        \lim_{r\to\infty}\int_0^{S}\frac{\sin(\theta r)}{\theta}\frac{\textup{d}\theta}{\pi}
        =\lim_{r\to\infty}\int_0^{rS}\frac{\sin(\theta)}{\theta}\frac{\textup{d}\theta}{\pi}
        =\lim_{r\to\infty}\texttt{Si}(r S)/\pi
        =1/2
        \end{align*}
        and remark that inserting the factor $\hat{\prob_{\e}}(\im\theta)$ in the above integrand does not change this limit, since $\hat{\prob_{\e}}$ is uniformly continuous and $\hat{\prob}_{\e}(0)=1$. Thus, for the third term on the right hand side of \eqref{eq:erg} we obtain
        \begin{align}\label{eq:Gx}       
        \lim_{r\to\infty}\Re\left(G(x)\int_{-\infty}^{\infty}\hat{\prob_{\e}}(\im\theta)\frac{\im\ee^{-\im\theta r}}{\theta}\frac{\textup{d}\theta}{2\pi}\right)
        =\frac{G(x)}{2}
        \end{align}
        uniformly for $x\in\Sigma_A$.
        Combining \eqref{eq:domconv}, \eqref{eq:erg}, \eqref{eq:Req} and \eqref{eq:Gx} it follows that 
        \begin{equation*}
        \lim_{r\to\infty}\int_{-\infty}^{\infty}\prob_{\e}(r-T)\ee^{-T\mdim}N(T,x)\textup{d}T=G(x)
        \end{equation*}
        uniformly for $x\in\Sigma_A$. An application of Lem.~\ref{lem:suffices} yields the desired result.
\end{proof}

\begin{proof}[Proof of Thm.~\ref{thm:RT1}\ref{it:RT1:l}]
In the lattice situation we work with discrete Fourier-Laplace transforms inspired by \cite[proof of Thm.~2]{Lalley}. Conditions \ref{it:boundedC} and \ref{it:boundedneg} imply that for fixed $\beta\in[0,\aaa)$ and $x\in\Sigma_A$, the function $\hat{N}^{\beta}(\cdot,x)$ given by 
\begin{equation}
	\hat{N}^{\beta}(z,x)\defeq\sum_{l=-\infty}^{\infty}\ee^{lz}N(\aaa l+\beta-\psi(x),x)
\end{equation}
is well-defined and analytic on 
\[
   \mathcal{Z}\defeq\{z\in\mathbb C\mid -\aaa(s+\mdim)<\Re(z)<-\aaa\mdim\}.
\]
Note that $S_n\xi=S_n\z+\psi-\psi\circ\sigma^n$ and recall that $S_n\z\in\aaa\mathbb Z$ for all $n\in\mathbb N$. Thus, \ref{it:boundedC} and \ref{it:boundedneg} imply that we can make the following conversions for $z\in\mathcal Z$. 

\begin{align*}
	\hat{N}^{\beta}(z,x)
	&=\sum_{n=0}^{\infty}\sum_{y:\sigma^ny=x}\codefunc(y)\ee^{S_n\eta(y)}
	\sum_{l=-\infty}^{\infty}\ee^{lz}\renfcn_{y}(\aaa l+\beta-\psi(x)-S_n\codefun(y))\\
        &=\sum_{n=0}^{\infty}\sum_{y:\sigma^ny=x}\codefunc(y)\ee^{S_n\eta(y)}
	\sum_{l=-\infty}^{\infty}\ee^{(l+\aaa^{-1}S_n\z(y))z}\renfcn_{y}(\aaa l+\beta-\psi(y))\\
        &=\sum_{n=0}^{\infty}\sum_{y:\sigma^ny=x}\codefunc(y)\ee^{S_n(\eta+\aaa^{-1}z\z)(y)}
	\sum_{l=-\infty}^{\infty}\ee^{lz}\renfcn_{y}(\aaa l+\beta-\psi(y))\\
        &=\sum_{n=0}^{\infty}\PF^n_{\eta+\aaa^{-1}z\z}\left(\codefunc
	\sum_{l=-\infty}^{\infty}\ee^{lz}\renfcn_{\cdot}(\aaa l+\beta-\psi)\right)(x)
\end{align*}
with $\renfcn_{\cdot}(t)\colon\Sigma_A\to\mathbb R$, $x\mapsto \renfcn_x(t)$ as before.
For $z\in\mathcal Z$ we have that $\Re(\aaa^{-1}z)<-\mdim$. Moreover, $\eigenf_{\eta-\mdim\z}=\ee^{-\mdim\psi}\eigenf_{\eta-\mdim\codefun}$ and $\eigenv_{\eta-\mdim\z}=\eigenv_{\eta-\mdim\codefun}$.
Thus, Prop.~\ref{thm:eigenvalueone} yields that $\eigenv_{\eta+\Re(\aaa^{-1}z)\z}<1$. Thm.~\ref{thm:TheoremBLalley} (both parts) and the spectral radius formula \eqref{thm:SpectralRadiusFormula} now imply, for every $z\in \mathcal Z$, that
\[
	\hat{N}^{\beta}(z,x)
	=(I-\PF_{\eta+\aaa^{-1}z\z})^{-1}\left(
\codefunc\sum_{l=-\infty}^{\infty} \ee^{lz}\renfcn_{\cdot}(\aaa l+\beta-\psi)\right)(x).
\]
Note that $\|\codefunc\sum_{l=-\infty}^{\infty} \ee^{lz}f_{\cdot}(\aaa l+\beta-\psi)\|_{\infty}$ is finite for all $z\in\mathcal Z$ because of Conditions \ref{it:boundedC} and \ref{it:boundedneg}.

Because $\aaa^{-1}\z$ is integer-valued but not co-homologous to any function valued in a proper subgroup of the integers, we can apply Prop.~\ref{prop:Lalley7.4}. Thus, $z\mapsto(I-\PF_{\eta+\aaa^{-1}z\z})^{-1}$ is $2\pi\im$-periodic and holomorphic at every $z$ on the line $\Re(z)=-\aaa\mdim$ such that $\Im(z)/(2\pi)$ is not an integer. Therefore, $z\mapsto (I-\PF_{\eta+\aaa^{-1}z\z})^{-1}$
has an isolated singularity at $z=-\aaa\mdim$ and is holomorphic at each $z=-\aaa\mdim+\im\theta$, for $0<\lvert\theta\rvert\leq\pi$. By Prop.~\ref{prop:Lalley7.12} the singularity of $\hat{N}^{\beta}(z,x)$ at $z=-\aaa\mdim$ is 
\begin{align*}
   \frac{\eigenv_{\eta+\aaa^{-1}z\z}}{1-\eigenv_{\eta+\aaa^{-1}z\z}}
\underbrace{\int_{\Sigma_A}\codefunc(y)\sum_{l=-\infty}^{\infty}\ee^{l z}\renfcn_y(\aaa l+\beta-\psi(y))\textup{d}\nu_{\eta+\aaa^{-1}z\z}(y)\eigenf_{\eta+\aaa^{-1}z\z}(x)}_{\eqdef E_{x,\beta}(z)}
\end{align*}
Since the function $E_{x,\beta}$ is continuous at $-\aaa\mdim$, we deduce from \eqref{eq:analyticpressure} that the singularity of $\hat{N}^{\beta}(z,x)$ at $z=-\aaa\mdim$ is a simple pole with residue
\begin{equation*}
        C_{\beta}(x)\defeq-\frac{\aaa}{\int\z\textup{d}\mu_{\eta-\mdim\z}}E_{x,\beta}(-\aaa\mdim).
\end{equation*}
It follows that $\hat{N}^{\beta}(\cdot,x)\colon\mathbb C\to\mathbb C$, $z\mapsto \hat{N}^{\beta}(z,x)$ is meromorphic in
\[
        \tilde{\mathcal Z}(\e)\defeq\{z\in\mathbb C\mid-\aaa(\mdim+s)<\Re(z)<-\aaa\mdim+\e,\ 0\leq\Im(z)\leq\pi\},
\]
for some $\e>0$, and that the only singularity in this region is a simple pole at $-\aaa\mdim$ with residue $C_{\beta}(x)$.
Additionally, by \ref{it:boundedneg}
\[
	\sum_{l=-\infty}^{-1}\ee^{lz}N(\aaa l+\beta-\psi(x),x)
\]
is finite for $\Re(z)>-\aaa(\mdim+ s)$. We conclude that there exists $\e>0$ such that
\[
	\sum_{l=0}^{\infty}\ee^{lz}N(\aaa l+\beta-\psi(x),x)-\frac{C_{\beta}(x)}{z+\aaa\mdim}
\]
is holomorphic in $\tilde{\mathcal{Z}}(\e)$. 
Also observe that $z\mapsto (\ee^{z+\aaa\mdim}-1)/(z+\aaa\mdim)$ is holomorphic in $\mathbb C$.
Making the change of variable $\tilde{z}\defeq\ee^{z+\aaa\mdim}$ we obtain that 
\[
	\sum_{l=0}^{\infty}\tilde{z}^l\ee^{-\aaa l\mdim}N(\aaa l+\beta-\psi(x),x)-\frac{C_{\beta}(x)}{\tilde{z}-1}
\]
is holomorphic in $\{\ee^{z+\aaa\mdim}\mid z\in\tilde{\mathcal{Z}}(\e)\}$. This implies that 
\[
	L(\tilde{z},x)\defeq\sum_{l=0}^{\infty}\tilde{z}^l\left(\ee^{-\aaa l\mdim}N(\aaa l+\beta-\psi(x),x)+C_{\beta}(x)\right)
\]
is holomorphic in $\{\tilde z\mid\lvert\tilde{z}\rvert<\ee^{\e}\}$ (compare \cite[p.\,27]{Lalley}). 
Since $\ee^{\e}>1$, the coefficient sequence of the power series of $L(\cdot,x)\colon\mathbb C\to\mathbb C$, $z\mapsto L(z,x)$ converges to zero exponentially fast, more precisely, 
\[
	\ee^{-\aaa n\mdim}N(\aaa n+\beta-\psi(x),x)+C_{\beta}(x)
	\in\mathfrak{o}((1+(\ee^{\e}-1)/2)^{-n})
\]
as $n\to\infty$ ($n\in\mathbb N$).
Thus, for $x\in\Sigma_A$ we have
\begin{align*}
	N(t,x)
	&=N\bigg(\aaa\underbrace{\left\lfloor \frac{t+\psi(x)}{\aaa}\right\rfloor}_{\eqdef n}+\underbrace{\aaa\Big{\{}\frac{t+\psi(x)}{\aaa}\Big{\}}}_{\eqdef\beta}-\psi(x),x\bigg)\\
	&\sim-\ee^{\aaa\left\lfloor\frac{t+\psi(x)}{\aaa}\right\rfloor\mdim}C_{\aaa\{(t+\psi(x))/{\aaa}\}}(x)\\
        &=\ee^{t\mdim}\ee^{-\aaa\big{\{}\frac{t+\psi(x)}{\aaa}\big{\}}\mdim}\ee^{\mdim\psi(x)}\frac{\aaa}{\int\z\textup{d}\mu_{\eta-\mdim\z}}\eigenf_{\eta-\mdim\z}(x)\\
        &\qquad\times\int_{\Sigma_A}\codefunc(y)\sum_{l=-\infty}^{\infty}\ee^{-l\aaa\mdim}\renfcn_y\left(\aaa l+\aaa\Big{\{}\frac{t+\psi(x)}{\aaa}\Big{\}}-\psi(y)\right)\textup{d}\nu_{\eta-\mdim\z}(y)\\
        &= \ee^{t\mdim}\tilde{G}_x(t)
\end{align*}
as $t\to\infty$. Since in all instances where $t$ occurs only the fractional part is involved, it is clear that $\tilde{G}_x$ is periodic with period $\aaa$, which finishes the proof.
\end{proof}

\begin{proof}[Proof of Thm.~\ref{thm:RT1}\ref{it:RT1:av}]
First, consider the case that $\codefun$ is non-lattice. Since $\ee^{-t\mdim}N(t,x)$ is bounded in $t$ by \ref{it:boundedC}, the result from Thm.~\ref{thm:RT1}\ref{it:RT1:nl} implies Thm.~\ref{thm:RT1}\ref{it:RT1:av}.

Second,  consider the case that $\codefun$ is lattice. Thm.~\ref{thm:RT1}\ref{it:RT1:l} states that
\begin{equation}
        \ee^{-t\mdim}N(t,x)
        \sim\tilde{G}_x(t)\quad\text{as}\ t\to\infty.
\end{equation}
Since $\ee^{-t\mdim}N(t,x)$ is bounded in $t$ by \ref{it:boundedC}, and $\tilde{G}_x$ is periodic with period $\aaa$ we have 
\begin{align*}
        \lim_{t\to\infty}t^{-1}\int_0^t\ee^{-T\mdim}N(T,x)\textup{d}T
        &=\lim_{t\to\infty}t^{-1}\int_0^{\aaa\left\lfloor t/a\right\rfloor}\ee^{-T\mdim}N(T,x)\textup{d}T\\
        &=\lim_{t\to\infty}t^{-1}\int_0^{\aaa\left\lfloor t/a\right\rfloor}\tilde{G}_x(T)\textup{d}T\\
        &=\lim_{t\to\infty}t^{-1}\left\lfloor \frac{t}{a}\right\rfloor\int_0^{\aaa}\tilde{G}_x(T)\textup{d}T\\
        &=\frac{1}{a}\int_0^{\aaa}\tilde{G}_x(T)\textup{d}T.
\end{align*}
The latter integral transforms as follows:
\begin{align*}
        &\int_0^{\aaa}\tilde{G}_x(T)\textup{d}T\\
        &=\int_0^{\aaa}\ee^{-\aaa\big{\{}\frac{T+\psi(x)}{\aaa}\big{\}}\mdim}
    \frac{\aaa\ee^{\mdim\psi(x)}}{\int\z\textup{d}\mu_{\eta-\mdim\z}}\eigenf_{\eta-\mdim\z}(x)\\
    &\quad\times\int_{\Sigma_A}\codefunc(y)\sum_{l=-\infty}^{\infty}\ee^{-\aaa l\mdim}\renfcn_y\left(\aaa l+\aaa\left\{\tfrac{T+\psi(x)}{\aaa}\right\}-\psi(y)\right)\textup{d}\nu_{\eta-\mdim\z}(y)\textup{d}T\\
    &=\frac{\aaa\ee^{\mdim\psi(x)}}{\int\z\textup{d}\mu_{\eta-\mdim\z}}\eigenf_{\eta-\mdim\z}(x)\int_{\Sigma_A}\codefunc(y)\\
    &\quad\times\sum_{l=-\infty}^{\infty}
\int_0^{\aaa}\ee^{-\left(\aaa l+\aaa\big{\{}\frac{T+\psi(x)}{\aaa}\big{\}}\right)\mdim}
\renfcn_y\left(\aaa l+\aaa\left\{\tfrac{T+\psi(x)}{\aaa}\right\}-\psi(y)\right)\textup{d}T\textup{d}\nu_{\eta-\mdim\z}(y)\\
    &=\frac{\aaa\ee^{\mdim\psi(x)}\eigenf_{\eta-\mdim\z}(x)}{\int\z\textup{d}\mu_{\eta-\mdim\z}}\int_{\Sigma_A}\hspace{-0.1cm}\codefunc(y)\ee^{-\psi(y)\mdim}
      \sum_{l=-\infty}^{\infty}\hspace{-0.1cm}
\int_{\aaa l-\psi(y)}^{\aaa (l+1)-\psi(y)}\hspace{-0.3cm}\ee^{-T\mdim}
\renfcn_y\left(T\right)\textup{d}T\textup{d}\nu_{\eta-\mdim\z}(y)\\
    &=\frac{\aaa\eigenf_{\eta-\mdim\codefun}(x)}{\int\codefun\textup{d}\mu_{\eta-\mdim\codefun}}\int_{\Sigma_A}\codefunc(y)
      \int_{-\infty}^{\infty}\ee^{-T\mdim}\renfcn_y\left(T\right)\textup{d}T\textup{d}\nu_{\eta-\mdim\codefun}(y)\\
      &=\aaa G(x),
\end{align*}
where for the second to last equality we used that $\ee^{-\mdim\psi}\textup{d}\nu_{\eta-\mdim\z}=\textup{d}\nu_{\eta-\mdim\xi}$,
$\ee^{\mdim\psi}\eigenf_{\eta-\mdim\z}=\eigenf_{\eta-\mdim\xi}$ and that $\int\z\textup{d}\mu_{\eta-\mdim\z}=\int\xi\textup{d}\mu_{\eta-\mdim\xi}$.
\end{proof}

%% file: corproof.tex
\subsubsection{Proof of Thm.~\ref{thm:RT2}}

In the setting of Thm.~\ref{thm:RT2}, Condition \ref{it:nldRi} is automatically satisfied. Thus, Thm.~\ref{thm:RT2} is proved by combining the following three lemmas with Thm.~\ref{thm:RT1}.

\begin{lemma}
  If $\{t\mapsto \ee^{-t\mdim}\lvert \renfcn_x(t)\rvert\mid x\in\Sigma_A\}$ is equi d.\,R.\,i.\ then \ref{it:Lebesgueintegrable} holds.
\end{lemma}
\begin{proof}
This is clear since d.\,R.\,i.\ implies Lebesgue integrability.
\end{proof}

\begin{lemma}\label{lem:dRiC}
	If $\{t\mapsto \ee^{-t\mdim}\lvert \renfcn_x(t)\rvert\mid x\in\Sigma_A\}$ is equi d.\,R.\,i.\ and there exist $\mathfrak{C}',s>0$ such that $\ee^{-t\mdim}\lvert\renfcn_x(t)\rvert\leq\mathfrak{C}'\ee^{st}$, for $t<0$ and $x\in\Sigma_A$ then \ref{it:boundedneg} holds, i.\,e.\ there exist $\tilde{\mathfrak C}>0, t^*\leq 0$ such that $\ee^{-\mdim t}N^{\text{abs}}(t,x)\leq \tilde{\mathfrak C}\ee^{st}$ for $t\leq t^*$.
\end{lemma}
\begin{proof}
In \eqref{eq:Snpositive} we have seen that the existence of $n\in\mathbb N$ for which $S_n\codefun$ is strictly positive implies existence of $m^*\in\mathbb N$ such that $S_m\codefun$ is strictly positive for all $m\geq m^*$.
Set 
\begin{equation*}
	t^*\defeq \min\bigg{\{}0,\inf_{x\in\Sigma_A,\,0\leq m\leq m^*}S_m\codefun(x)\bigg{\}}.
\end{equation*}
Then $t-S_m\codefun<0$ for all $t< t^*$ and $m\in\mathbb N$. Additionally using that there exist $\mathfrak{C}',s>0$ such that $\ee^{-t\mdim}\lvert\renfcn_x(t)\rvert\leq\mathfrak{C}'\ee^{st}$ for $t<0$ and $x\in\Sigma_A$ we have for $t\leq t^*$,
\begin{align*}
  \ee^{-\mdim t}N^{\text{abs}}(t,x)
  &= \sum_{n=0}^{\infty}\sum_{y:\sigma^ny=x}\ee^{S_n(\eta-\mdim\codefun)(y)}\lvert\codefunc(y)\rvert\ee^{-\mdim(t-S_n\codefun(y))}\lvert \renfcn_y(t-S_n\codefun(y))\rvert\\
  &\leq \mathfrak C'\ee^{st}\sum_{n=0}^{\infty}\sum_{y:\sigma^ny=x}\ee^{S_n(\eta-(\mdim+s)\codefun)(y)}\lvert\codefunc(y)\rvert\\
  &=\mathfrak C'\ee^{st}\sum_{n=0}^{\infty}\PF_{\eta-(\mdim+s)\codefun}^n\lvert\codefunc\rvert(x)\\
  &\leq\mathfrak C'\ee^{st}\sum_{n=0}^{\infty}\|\PF_{\eta-(\mdim+s)\codefun}^n\|_{\textup{op}}\|\codefunc\|_{\infty}.
\end{align*}
By Prop.~\ref{thm:eigenvalueone} we know that $t\mapsto\eigenv_{\eta+t\codefun}$ is strictly monotonically increasing and $\eigenv_{\eta-\mdim\codefun}=1$. Hence, by the spectral radius formula \eqref{thm:SpectralRadiusFormula} the last series converges and the assertion follows.
\end{proof}

\begin{lemma}
        If $\{t\mapsto \ee^{-t\mdim}\lvert\renfcn_x(t)\rvert\mid x\in\Sigma_A\}$ is equi d.\,R.\,i.\ then \ref{it:boundedC} is satisfied, i.\,e.\ there exists  $\mathfrak C>0$ such that $\ee^{-t\mdim}N^{\text{abs}}(t,x)\leq \mathfrak C$ for all $t\in\mathbb R$.
\end{lemma}
\begin{proof}
  By assumption there exists $n\in\mathbb N$ for which $S_n\codefun$ is strictly positive. Fix this $n$, choose $\kappa>0$ such that $S_n\codefun\geq\kappa$ and consider the $n$-th iterate of the renewal equation, namely,
  \begin{align*}
    N(t,x)
    =\hspace{-0.3cm} \sum_{y:\sigma^n y=x}\hspace{-0.2cm}N(t-S_n\codefun(y),y)\ee^{S_n\eta(y)}
      + \sum_{i=0}^{n-1}\sum_{y:\sigma^iy=x}\hspace{-0.2cm}\codefunc(y)\renfcn_y(t-S_i\codefun(y))\ee^{S_i\eta(y)}.
  \end{align*}
  The function defined by $M(t,x)\defeq\ee^{-t\mdim}N(t,x)/\eigenf_{\eta-\mdim\codefun}(x)$ satisfies
  \begin{align}
    M(t,x)
    &= \sum_{y:\sigma^n y=x}M(t-S_n\codefun(y),y)\ee^{S_n(\eta-\mdim\codefun)(y)}\frac{\eigenf_{\eta-\mdim\codefun}(y)}{\eigenf_{\eta-\mdim\codefun}(x)}\label{eq:renM} \\
    &\quad+ \sum_{i=0}^{n-1}\sum_{y:\sigma^iy=x}\codefunc(y)\ee^{-\mdim(t-S_i\codefun(y))}\renfcn_y(t-S_i\codefun(y))\ee^{S_i(\eta-\mdim\codefun)(y)}\frac{1}{\eigenf_{\eta-\mdim\codefun}(x)}.\nonumber
  \end{align}
Set $\renfcnn_x(t)\defeq\ee^{-t\mdim}\renfcn_x(t)$ and $\renfcnn_{i,x}(t)\defeq \renfcnn_x(t-S_i\codefun(x))$ for $x\in\Sigma_A$. Then \eqref{eq:renM} becomes
\begin{align}
\begin{aligned}\label{eq:renMPF}
    M(t,x)
    &= \sum_{y:\sigma^n y=x}M(t-S_n\codefun(y),y)\ee^{S_n(\eta-\mdim\codefun)(y)}\frac{\eigenf_{\eta-\mdim\codefun}(y)}{\eigenf_{\eta-\mdim\codefun}(x)}\\
    &\qquad+ \sum_{i=0}^{n-1}\PF_{\eta-\mdim\codefun}^i(\codefunc\renfcnn_{i,\cdot}(t))(x)/\eigenf_{\eta-\mdim\codefun}(x)
\end{aligned}
\end{align}
  with $\renfcnn_{i,\cdot}(t)\colon\Sigma_A\to\mathbb R$, $x\mapsto \renfcnn_{i,x}(t)$. Define
  \begin{equation}
    \overline{M}(t)\defeq\sup_{\substack{t'\in(t-\kappa,t]\\x\in\Sigma_A}}M(t',x).
  \end{equation}
  Then \eqref{eq:renMPF} implies that
  \begin{align}
    \overline{M}(t)
    \leq \sup_{y\in\Sigma_A}\overline{M}(t-S_n\codefun(y))
    +\sup_{t'\in(t-\kappa,t]}\sum_{i=0}^{n-1}\|\PF^i_{\eta-\mdim\codefun}(\codefunc\renfcnn_{i,\cdot}(t'))/\eigenf_{\eta-\mdim\codefun}\|_{\infty},\label{eq:Mbar}
  \end{align}
as $\sum_{y:\sigma^n y=x}\ee^{S_n(\eta-\mdim\codefun)(y)}\eigenf_{\eta-\mdim\codefun}(y)/\eigenf_{\eta-\mdim\codefun}(x)=1$.
Iterating \eqref{eq:Mbar} $k$ times and using the abbreviation 
$I_{\kappa}(t,x^1,\ldots,x^m)\defeq (t-\sum_{j=1}^m S_n\codefun(x^j)-\kappa,t-\sum_{j=1}^m S_n\codefun(x^j)]$ 
yields that
\begin{align*}
  \overline{M}(t)
  &\leq \sup_{x^1,\ldots,x^k\in\Sigma_A}\bigg{[}\overline{M}\Big{(}t-\sum_{j=1}^k S_n\codefun(x^j)\Big{)}\\
  &\quad+\sum_{i=0}^{n-1}\sum_{m=0}^{k-1}\sup_{t'\in I_{\kappa}(t,x^1,\ldots,x^m)}\|\PF^i_{\eta-\mdim\codefun}(\codefunc\renfcnn_{i,\cdot}(t')/\eigenf_{\eta-\mdim\codefun}\|_{\infty}\bigg{]},
\end{align*}
where $\sum_{j=1}^0 S_n\codefun(x^j)$ shall be understood to be zero.

Fix $t\in\mathbb R$. As $S_n\codefun\geq\kappa>0$ we can find a $k^*\in\mathbb N$ such that for all $k\geq k^*$ and $x^1,\ldots, x^k\in\Sigma_A$ we have $t-\sum_{j=1}^k S_n\codefun(x^j)\leq t^*\leq 0$ with $t^*$ as in Lem.~\ref{lem:dRiC}. Hence by Lem.~\ref{lem:dRiC} there exists a constant $\mathfrak C''$ such that 
\begin{align}\label{eq:MbarC''}
  \overline{M}(t)
  &\leq \mathfrak C''
  +\sup_{x^1,x^2,\ldots\in\Sigma_A}\sum_{i=0}^{n-1}\sum_{m=0}^{\infty}\sup_{t'\in I_{\kappa}(t,x^1,\ldots,x^m)}\|\PF^i_{\eta-\mdim\codefun}(\codefunc\renfcnn_{i,\cdot}(t'))/\eigenf_{\eta-\mdim\codefun}\|_{\infty}.
\end{align}

If $t'\mapsto\|\PF_{\eta-\mdim\codefun}^i(\codefunc\renfcnn_{i,\cdot}(t'))/\eigenf_{\eta-\mdim\codefun}\|_{\infty}$ is d.\,R.\,i.\  then $t-\sum_{j=1}^{m+1}S_n\codefun(x^j)\leq t-\sum_{j=1}^{m}S_n\codefun(x^j)-\kappa$ for $x^1,\ldots,x^{m+1}\in\Sigma_A$ implies that the series in \eqref{eq:MbarC''} converges for any constellation  $x^1, x^2,\ldots\in\Sigma_A$ and thus is uniformly bounded for $x^1, x^2,\ldots\in\Sigma_A$. This proves the assertion. Hence all that remains to be shown is that $t'\mapsto\|\PF_{\eta-\mdim\codefun}^i(\codefunc\renfcnn_{i,\cdot}(t'))/\eigenf_{\eta-\mdim\codefun}\|_{\infty}$ is d.\,R.\,i. For this, note that using the terminology of Defn.~\ref{def:dRi}
\begin{align*}
        \underline{m}_k\left(\PF_{\eta-\mdim\codefun}^i(\codefunc\renfcnn_{i,\cdot})(x),h\right)
        &\geq\sum_{y:\sigma^iy=x}\codefunc(y)\ee^{S_i(\eta-\mdim\codefun)(y)}\underline{m}_{k-S_i\codefun(y)/h}(\renfcnn_y,h)\quad\text{and}\\
        \overline{m}_k\left(\PF_{\eta-\mdim\codefun}^i(\codefunc\renfcnn_{i,\cdot})(x),h\right)
        &\leq\sum_{y:\sigma^iy=x}\codefunc(y)\ee^{S_i(\eta-\mdim\codefun)(y)}\overline{m}_{k-S_i\codefun(y)/h}(\renfcnn_y,h).
\end{align*}
  Since $\sum_{y:\sigma^iy=x}\codefunc(y)\ee^{S_i(\eta-\mdim\codefun)(y)}=\PF^i_{\eta-\mdim\codefun}\codefunc(x)$ is finite, the hypothesis of $\{\renfcnn_x\mid x\in\Sigma_A\}$ being equi d.\,R.\,i.\ implies that $t'\mapsto\|\PF_{\eta-\mdim\codefun}^i(\codefunc\renfcnn_{i,\cdot}(t'))/\eigenf_{\eta-\mdim\codefun}\|_{\infty}$ is d.\,R.\,i.\ which finishes the proof.
\end{proof}